\documentclass[12pt,reqno]{amsart}
\usepackage[latin1]{inputenc}
\usepackage{amscd}
\usepackage{amsmath,amsthm,amssymb,amsfonts}
\usepackage{mathrsfs}
\usepackage{graphicx}
\usepackage{fancyhdr}
\usepackage{stmaryrd}

\usepackage{color}
\usepackage{amsmath}
\usepackage{amsfonts}
\usepackage{amssymb}
\usepackage{geometry}
\numberwithin{equation}{section}

\newtheorem{theorem}{Theorem}[section]
\newtheorem{proposition}{Proposition}[section]
\newtheorem{lemma}{Lemma}[section]

\theoremstyle{Definition}

\newtheorem{example}{Example}[section]

\def \d{\delta}
\def\R{\mathbb{R}}

\def \p{\partial}
\def \s{\sigma}

\def \e {\varepsilon}
\def \I {\mathcal I}

\def \a {\alpha}
\def \RE {\text{Re}}
\def \IM {\text{Im}}
\def \O {\Omega}
\def \L {\mathcal{L}}
\def \B {\mathcal{B}}

\DeclareMathOperator{\dist}{dist}

\title[Variational reduction for semi-stiff GL vortices]{Variational reduction for semi-stiff Ginzburg-Landau vortices}
\author{R\'emy Rodiac}
\date{}
\address[R.Rodiac]{Facultad de Matematic\'as, Pontificia Universidad Cat\'olica de Chile, Vicu\~na Mackenna 4860, Macul, Santiago Chile and Universit\'e catholique de Louvain, Institut de Recherche en Math\'e-
matique et Physique, Chemin du Cyclotron 2 bte L7.01.01, 134
8 Louvain-la-Neuve, Belgium}
\email{remy.rodiac@mat.uc.cl, remy.rodiac@uclouvain.be}

\begin{document}

\begin{abstract}
Let $\O$ be a smooth bounded domain in $\R^2$. For $\e>0$ small, we construct non-constant solutions to the Ginzburg-Landau equations $-\Delta u=\frac{1}{\e^2}(1-|u|^2)u$ in $\O$ such that on $\p \O$ u satisfies $|u|=1$ and $u\wedge \p_\nu u=0$. These boundary conditions are called semi-stiff and are intermediate between the Dirichlet and the homogeneous Neumann boundary conditions. In order to construct such solutions we use a variational reduction method very similar to the one used in \cite{delPinoKowalczykMusso2006}. We obtain the exact same result as the authors of the aforementioned article obtained for the Neumann problem. This is because the renormalized energy for the Neumann problem and for the semi-stiff problem are the same. In particular if $\O$ is simply connected a solution with degree one on the boundary always exists and if $\O$ is not simply connected then for any $k\geq 1$ a solution with $k$ vortices of degree one exists.


\end{abstract}

\maketitle

\section{Introduction}

Let $\O \subset \R^2$ be a smooth bounded domain. The aim of the present paper is to prove existence of solutions $u:\O \rightarrow \mathbb{C}$ of the Ginzburg-Landau (GL) equations
\begin{equation}\label{equation}
-\Delta u=\frac{1}{\e^2}(1-|u|^2)u \text{ in } \O,
\end{equation}
with the following boundary conditions

\begin{equation}\label{BC}
\left\{
\begin{array}{rcll}
|u| &=& 1  \text{ on } \p \O,\\
u\wedge \p_\nu u &=& 0 \text{ on } \p \O.
\end{array}
\right.
\end{equation}
Here $\e>0$ is a (small) parameter, $\nu$ is the outer unit normal to $\p \O$ and $\p_\nu$ denotes the normal derivative of $u$ on $\p \O$. For $a,b$ in $\mathbb{C}$, 
\begin{equation}
a\wedge b:=\frac{1}{2i}(\overline{a}b-a\overline{b})=\IM(\overline{a}b)
\end{equation} is the determinant of $a$ and $b$ viewed as vectors in $\R^2$.

This problem is called semi-stiff because it is intermediate between the Dirichlet problem, called the stiff problem in the Ginzburg-Landau literature, and the Neumann problem, refereed as the soft problem. Indeed $|u|=1$ is a Dirichlet condition for the modulus and $u\wedge \p_\nu u=0$ is a homogeneous Neumann condition for the phase. This is because, assuming that $u$ is smooth near the boundary, locally we can write $u=|u|e^{i\psi}$, where $\psi$ is a phase of $u$. A direct computation shows that $u\wedge \p_\nu u =|u|\p_\nu \psi$. Thus the second condition of \eqref{BC} leads to $\p_\nu \psi=0$ on the boundary.

Solutions of \eqref{equation}--\eqref{BC} are critical points of the energy
\begin{equation}\label{energy}
E_\e(u)=\frac12\int_\O |\nabla u|^2+\frac{1}{4\e^2} (1-|u|^2)^2,
\end{equation}
in the space
\begin{equation}
\I=\{ v\in H^1(\O,\mathbb{C}); \ |u|=1 \ \text{ on } \p \O \}.
\end{equation}
The space $\I$ is not a vector space and it is not clear if it is a Hilbert manifold. That is why we should precise the definition of critical points. We say that $u$ is a critical point of $E_\e$ in $\I$ if for all $\varphi \in C^\infty_c(\O,\R^2)$ and all $\psi \in C^\infty(\overline{\O},\R)$, setting $u_t=u+t\varphi$ and $\tilde{u_t}=ue^{it\psi}$ we have
\[ \frac{d}{dt}|_{t=0}E_\e(u_t)=0 \text{ and  } \frac{d}{dt}|_{t=0} E_\e(\tilde{u_t})=0. \]

The Ginzburg-Landau equations were extensively studied in the past decades. The Ginzburg-Landau model is used to describe the behavior of superconductor materials. These materials can be divided into two categories: superconductors of type I (which correspond to $\e$ large in the model) and superconductors of type II (here $\e$ is small). An interesting feature of type II-superconductors is the apparition of vortices (i.e.\ small regions where the material is in the normal state surrounded by rotating superconducting current). These vortices are due to the presence of a magnetic field. However, in their pioneer work \cite{BethuelBrezisHelein94}, Bethuel-Brezis-Helein observed that we can study vortices in the absence of magnetic field if we prescribe a Dirichlet boundary data with non zero topological degree on the boundary of the domain. If $\Gamma$ is a smooth, simple, connected curve, and $g$ is in $C^1(\Gamma,\mathbb{S}^1)$ then the degree (or winding number) of $g$ is defined by
\begin{equation}
\deg(g,\Gamma)=\frac{1}{2\pi}\int_{\Gamma} u\wedge \p_\tau u d\tau.
\end{equation}
This is an integer which measures the algebraic variation of the phase of $u$. Boutet-de-Monvel--Gabber observed in \cite[Appendix]{BoutetGeorgescuPurice91}, that we can extend the definition of the topological degree for maps $g$ in $H^\frac{1}{2}(\Gamma,\mathbb{S}^1)$. This can be done by approximation since smooth maps are dense in $H^\frac{1}{2}(\Gamma,\mathbb{S}^1)$ and the degree is continuous for the strong convergence in that space. Thus maps in $\I$ have a well-defined degree on every connected components of $\p \O$. Bethuel-Brezis-Helein observed that prescribing a Dirichlet boundary conditions is not physically realistic since in superconductivity theory only $|u|^2$ has a physical meaning and not $u$ ($|u|^2$ represents the density of Cooper-pairs of electrons). The degree also has a physical meaning because it describes the vorticity, i.e.\ a measure of how the superconducting currents rotates. That is why it seems natural to try to find critical points of the Ginzburg-Landau energy with modulus one on the boundary and with prescribed degrees.

For the moment we take $\O$ simply connected, then we can decompose
\[\I=\bigcup_{d\in \mathbb{N}} \I_d, \]
where
\[\I_d= \{v\in \I; \ \deg(v,\p \O)=d \}.\]
A result of Boutet-de-Monvel--Gabber \cite[Appendix]{BoutetGeorgescuPurice91} shows that the spaces $\I_d$ are the connected components of $\I$ and they are open and closed for the strong topology of $H^1$. Thus minimizing the energy $E_\e$ in an $\I_d$ would provide a critical point of $E_\e$ in $\I$ and hence a solution of \eqref{equation}--\eqref{BC}. However since the degree is not continuous for the weak $H^\frac{1}{2}$ convergence we can not apply the direct method of the calculus of variations to find minimizers. The following concentration phenomenon can occur:

\begin{example}
Let $M_n:\mathbb{D} \rightarrow \mathbb{D}$ be defined by  $M_n(z)= \frac{z-(1-1/n)}{1-(1-1/n)z}$, then $M_n \rightharpoonup -1$
weakly in $H^1$,
$\deg(M_n(z),\mathbb{S}^1)=1$   for all $n \in \mathbb{N}$ but $\deg(-1,\mathbb{S}^1)=0$.
\end{example}

Thus we might face a problem of lack of compactness. It has been shown, see e.g.\ \cite[Lemma 3.4]{BerlyandMironescuRybalkoSandier2014}, that if $\O$ is simply connected the infimum of $E_\e$ in $\I_d$, for $\e>0$ and $d\neq 0$ is not attained. In multiply connected domain the existence of minimizers depends on various factors and there is a delicate interplay between the parameter $\e$, the capacity of the domain (see e.g.\ \cite{BerlyandMironescu2006} for the definition of capacity) and the values of the prescribed degrees (see \cite{GolovatyBerlyand2002}, \cite{BerlyandMironescu2006}, \cite{BerlyandRybalkoGolovaty2006}, \cite{FM2013}, \cite{Mironescu2013}, \cite{DosSantosRodiac2016}). However for $\e$ small and for multiply connected domains, critical points with vortices exist for every values of prescribed degrees (see \cite{BerlyandRybalko2010}, \cite{DosSantos2009}, and also \cite{RodiacSandier2014}). The critical points constructed in the previous articles are stable (they are local minimizers in an appropriate function space) and their vortices are at a distance $o(\e)$ of the boundary. Thus they are near-boundary vortices and ''escape" the domain as $\e$ goes to $0$.
In simply connected domains the existence of critical points of $E_\e$ in $\I_1$ has been shown in \cite{BerlyandMironescuRybalkoSandier2014} for $\e$ large and in \cite{LamyMironescu2014} for $\e$ small. The paper \cite{BerlyandMironescuRybalkoSandier2014} rests upon a mountain-pass approach and a bubbling analysis of Palais-Smale sequences (in the spirit of \cite{BrezisCoron85}, \cite{Struwe1984}), whereas the paper \cite{LamyMironescu2014} uses singular perturbations techniques in the spirit of Pacard-Riviere \cite{PacardRiviereBirk2000}. This last approach relies on some non-degeneracy condition of the domain. One of the goals of the present paper is to get rid of this non-degeneracy condition. In order to do that we follow the approach of del Pino-Kowalczyk-Musso in \cite{delPinoKowalczykMusso2006}. Before explaining this approach we recall some results on the asymptotic behavior as $\e \rightarrow 0$ of solutions of \eqref{equation}.

In \cite{BethuelBrezisHelein94}, the authors studied  the asymptotic behavior for Dirichlet boundary data $g:\p \O \rightarrow \mathbb{S}^1$. They assumed that $\O$ is starshaped, $g$ is smooth and that $\deg(g,\p \O)=d>0$. Then they established that for a given family of solutions $u_\e$ there exist a number $k\geq 1$, and $k$-tuples
\[ \xi=(\xi_1,...,\xi_k)\in \O^k, \ \ \mathtt{d}=(d_1,d_2,...,d_k)\in \mathbb{Z}^k, \]
with $\xi_i\neq \xi_j$ and $\sum_{j=1}^k d_j=d$, such that $u_\e(x) \rightarrow w_g(x,\xi,\mathtt{d})$ along a suitable subsequence, in $C^1$-sense away from the vortices $\xi_j$, where
\begin{equation}\nonumber
w_g(x,\xi,\mathtt{d})=e^{i\varphi_g(x,\xi,\mathtt{d})} \prod_{j=1}^k \left( \frac{x-\xi_j}{|x-\xi_j|}\right)^{d_j}.
\end{equation}
Here the products are understood in complex sense and $\varphi_g=\varphi_g(x,\xi,\mathtt{d})$ is the unique solution of the problem
\begin{equation}\nonumber
\left\{
\begin{array}{rclll}
\Delta \varphi_g &=&0  &\text{ in } \O, \\
w_g(x,\xi,d)&=&g(x) & \text{ on } \p \O.
\end{array}
\right.
\end{equation}
Besides, $\xi$ must be a critical point of a \textit{renormalized energy}, $W_g(\xi,\mathtt{d})$, characterized as the limit
\begin{equation}\nonumber
W_g(\xi,\mathtt{d})=\lim_{\rho \rightarrow 0} \left[ \int_{\O \setminus \bigcup_{j=1}^k B_\rho(\xi_j)} |\nabla_x w_g|^2dx-\pi\sum_{j=1}^k d_j^2\log\frac{1}{\rho} \right].
\end{equation}
Explicit expressions in terms of Green's functions can be found in \cite{BethuelBrezisHelein94}. The renormalized energy also arises through a slightly different approach. We take $\O$ simply connected, $\xi\in \O^k$ and $\mathtt{d}\in \mathbb{Z}^k$ as before. We also set $\O_\rho=\O \setminus \bigcup_{i=1}^k\overline{B(a_i,\rho)}$ and we consider the space
\begin{equation}\nonumber
\mathcal{E}_{\rho,g}=\{v \in H^1(\O_\rho,\mathbb{S}^1); \ v=g \text{ on } \p \O \text{ and } \deg(v,\p B(\xi_i,\rho))=d_i, \text{ for } i=1,...,k\}
\end{equation}
and the minimization problem
\begin{equation}\nonumber
E_{\rho,g}= \inf\left\{ \int_{\O_\rho} |\nabla v|^2; v\in \mathcal{E}_{\rho,g} \right\}.
\end{equation}
Then we have (see \cite{BethuelBrezisHelein94}) that
\begin{equation}\nonumber
E_{\rho,g}=\pi\left(  \sum_{i=1}^k d_i^2 \right) \log \frac{1}{\rho}+W_g(\xi,\mathtt{d}) +O(\rho), \text{ as } \rho\rightarrow 0.
\end{equation}
Thus in order to derive a renormalized energy for the semi-stiff problem it is natural to set
\begin{equation}\nonumber
\mathcal{F}_\rho=\{v \in H^1(\O_\rho,\mathbb{S}^1); \deg(v,\p \O)=d \text{ and } \deg(v,\p B(\xi_i,\rho))=d_i, \text{ for } i=1,...,k\}
\end{equation}
and
\begin{equation}\nonumber
F_\rho=\inf\left\{ \int_{\O_\rho} |\nabla v|^2; v\in \mathcal{F}_{\rho} \right\}.
\end{equation}
As shown in \cite{LefterRadulescu96} a similar asymptotic expansion holds
\begin{equation}\nonumber
F_\rho=\pi\left( \sum_{i=1}^k d_i^2 \right) \log \frac{1}{\rho}+ W_N(\xi,d)+O(\rho) \text{ as } \rho \rightarrow 0.
\end{equation}
for some quantity $W_N(\xi,\mathtt{d})$. In order to give an explicit expression of $W_N$ we introduce $\hat{\Phi}_\xi$ the unique solution of
\begin{equation}\nonumber
\left\{
\begin{array}{rclll}
\Delta \hat{\Phi}_\xi &=& 2 \pi \sum_{i=1}^k d_i\delta _{\xi_i} & \text{ in }  \O, \\
\hat{\Phi}_\xi &=& 0 & \text{ on } \p \O.
\end{array}
\right.
\end{equation}
Then we have (cf.\ Theorem 2 of \cite{LefterRadulescu96})

\begin{equation}\nonumber
W_N(\xi,\mathtt{d})=-\pi \sum_{i\neq j} d_id_j \log|\xi_i-\xi_j|-\pi \sum_{j=1}^k d_j \hat{R}_\xi(\xi_j),
\end{equation}
with
\begin{equation}\nonumber
\hat{R}_\xi(x)=\hat{\Phi}_\xi(x)-\sum_{j=1}^kd_j\log|x-\xi_j|.
\end{equation}

The asymptotic behavior of solutions of \eqref{equation}--\eqref{BC} was not studied because the existence of solutions was not clear. However Lefter-Radulescu studied the asymptotic behavior of solutions of \eqref{equation} in a subclass of $\I_d$ for every $d\neq 0$. For $A>0$ they set
\begin{equation}\nonumber
\I_{d,A}=\{v\in \I_d; \int_{\p \O} |\p_\tau u|^2 \leq A \}.
\end{equation}
For $A>0$ large enough this set is not empty and minimizers of $E_\e$ in $\I_{d,A}$ exist. They found a similar asymptotic behavior of these minimizers as in \cite{BethuelBrezisHelein94} but with the new renormalized energy $W_N(\xi,d)$. Note that the minimizers found in \cite{LefterRadulescu96} are not solutions of \eqref{BC}, because it is not true that $\tilde{u_t}=ue^{it\psi}$ belongs to $\I_{d,A}$ for $t$ small and $\psi$ in $C^\infty(\overline{\O},\R)$.

Asymptotic behavior of solutions of GL equations with homogeneous Neumann boundary conditions were also studied in \cite{Serfaty2005} (see also \cite{Spirn2003}). A renormalized energy was derived and we want to emphasize that it is the same as in the paper of Lefter-Radulescu. The heuristic reason for that is that when $\e$ goes to $0$  solutions of the GL equations tends to singular $\mathbb{S}^1$-valued harmonic maps. We expect solutions of the GL equations with Neumann boundary conditions to tend to  $\mathbb{S}^1$-valued harmonic maps with Neumann boundary conditions. However for $u$ an $\mathbb{S}^1$-valued map a homogeneous Neumann boundary condition for the map translates in a homogeneous Neumann boundary condition for the phase, that is $u\wedge \p_\nu u=0$ on the boundary. Hence formally, for solutions with Neumann boundary conditions we find a limit with semi-stiff boundary condition, and that same limit is expected for limit of solutions with semi-stiff boundary conditions.

The notion of renormalized energy is also useful to construct solutions of GL equations with various boundary conditions. In \cite{PacardRiviereBirk2000} the authors constructed such solutions with Dirichlet boundary conditions and with vortices which converge to non-degenerate critical points of the corresponding renormalized energy. In \cite{delPinoKowalczykMusso2006}, the authors constructed solutions of \eqref{equation} both with Dirichlet and Neumann boundary condition and without using any non-degeneracy asumption. Since the semi-stiff boundary condition are intermediate between Dirichlet and Neumann, and since the renormalized energy for this problem is the same as in the Neumann case, it is natural to try to adapt the method of \cite{delPinoKowalczykMusso2006} to construct solutions of \eqref{equation}--\eqref{BC}. This is the goal of this paper.

In order to state our main results we rewrite in a slightly different form the renormalized energy in the semi-stiff case (and that is the same as in the Neumann case). For $\xi,\mathtt{d}$ as before we let
\begin{equation}\nonumber
w_N(x,\xi,\mathtt{d})=e^{i\varphi_N(x,\xi,\mathtt{d})} \prod_{j=1}^k \left( \frac{x-\xi_j}{|x-\xi_j|} \right)^{d_j},
\end{equation}
with $\varphi_N(x,\xi,\mathtt{d})$ the unique solution of the following problem:
\begin{equation}\label{varphiprblem}
\Delta \varphi_N = 0 \ \ \text{ in } \O,
\end{equation}
\begin{equation}\label{varphiproblem2}
\p_\nu \varphi_N=-\sum_{j=1}^k d_j \frac{(x-\xi_j)^\perp\cdot \nu}{|x-\xi_j|^2} \ \ \text{ on } \p \O, \ \ \int_\O \varphi_N=0.
\end{equation}

We have let $x^\perp=(-x_2,x_1)$. We can check that, if $\O$ is simply connected, then (cf.\ \cite{delPinoKowalczykMusso2006})
\begin{equation}\label{Renormalizedenergy}
W_N(\xi,\mathtt{d})= \lim_{\rho \rightarrow 0} \left[ \int_{ \O \setminus \bigcup_{j=1}^k B_\rho(\xi_j)} |\nabla_x w_N|^2-\pi \sum_{j=1}^k d_j^2 \log \frac {1}{\rho} \right].
\end{equation}
If $\O$ is multiply-connected, then we take \eqref{Renormalizedenergy} as a definition of the renormalized energy. We also introduce the standard single vortex solutions $w_{\pm}$ of respective degrees $+1$ and $-1$. These are solutions of
\begin{equation}\nonumber
\Delta w+(1-|w|^2)w=0 \ \text{ in } \R^2,
\end{equation}
which have the form
\begin{equation}\nonumber
w_+(x)=U(r)e^{i\theta}, \ \ \ w_-(x)=U(r)e^{-i\theta},
\end{equation}
where $(r,\theta)$ denote the polar coordinates ($x_1=r\cos \theta$, $x_2=r\sin \theta$) and $U(r)$ is the unique solution of the problem
\begin{equation}\label{equationradiale}
\left\{
\begin{array}{rcll}
U''+\frac{U'}{r}-\frac{U}{r^2}+(1-|U|^2)U&=&0 \ \ \text{ in } (0,+\infty), \\
U(0)=0, \ \ \ U(+\infty)=1. & &
\end{array}
\right.
\end{equation}
It is known, see e.g.\  \cite{ChenElliottQi1994}, that $U'(0)>0$,
\begin{equation}\label{estimeessurlesderivees}
U(r)=1-\frac{1}{2r^2}+O(\frac{1}{r^4}), \ \  U'(r)=\frac{1}{r^3}+O(\frac{1}{r^4})  \text{ as } r\rightarrow +\infty.
\end{equation}
Besides by using the equation \eqref{equationradiale} and \eqref{estimeessurlesderivees} we also have 
\begin{equation}\label{estimeesurladeriveeseconde}
U''(r)=O(\frac{1}{r^4}) \text{ as } r\rightarrow +\infty.
\end{equation}
Let us fix a number $k\geq 1$, and sets $I_{\pm}$ with
\begin{equation}\nonumber
I_-\cup I_+ =\{1,...,k\}, \ \ \ I_+\cap I_-=\emptyset.
\end{equation}
Let $\xi=(\xi_1,...,\xi_k)$ be a $k$-tuple of distinct points of $\O$, and
\begin{equation}\nonumber
\mathtt{d} \in \{ -1,1\}^k, \ \ d_j=\pm 1 \ \text{ if } j\in I_{\pm}.
\end{equation}
As an approximation of a solution of \eqref{equation}--\eqref{BC} we consider
\begin{equation}\label{approximation}
w_{N\e}(x,\xi,\mathtt{d})=e^{i\varphi_N(x,\xi,\mathtt{d})} \prod_{j\in I_+}w_+\left( \frac{x-\xi_j}{\e}\right) \prod_{j\in I_{-}} w_{-}\left( \frac{x-\xi_j}{\e}\right),
\end{equation}
where the products are understood to be equal to one if $I_-$ of $I_+$ are empty. We note that the approximation is the same as the one used for the Neumann problem in \cite{delPinoKowalczykMusso2006}. In this article we build solutions of \eqref{equation}--\eqref{BC} which are close (uniformly) to $w_{N\e}$. For that we use a technique called \emph{variational reduction} (or Lyapounov-Schmidt approach) which consists in building a finite-dimensional manifold of approximate solutions such that critical points of $E_\e$  constrained to this manifold correspond to actual solutions of \eqref{equation}--\eqref{BC}. In our case that manifold is parametrized by the locations of the vortices of approximate solutions.

We say that $W_N(\cdot,\mathtt{d})$ satisfies a non-trivial critical point situation in $\mathcal{D}$, open and bounded subset of $\O^k$ with
\[ \overline{\mathcal{D}}\subset \{ \xi \in \O^k ; \xi_i\neq \xi_j, \ \text{ if } i\neq j \},\]
if there exists a $\delta>0$ such that for any $h$ in $C^1(\overline{\mathcal{D}})$ with $\|h\|_{C^1(\overline{\mathcal{D}})}<\delta$, a critical point for $W_N+h$ in $\mathcal{D}$ exists. As in \cite{delPinoKowalczykMusso2006} we have

\begin{theorem}\label{theorem1}
Assume that $W_N$ exhibits a non-trivial critical point situation in $\mathcal{D}$. Then there exists a solution $u_\e$ to the semi-stiff problem \eqref{equation}--\eqref{BC} such that
\begin{equation}\label{1.21}
u_\e(x)=w_{N\e}(x,\xi_\e,\mathtt{d})+o(1),
\end{equation}
where $o(1)\rightarrow 0$ uniformly in $\O$ and
\begin{equation}\label{1.22}
\xi_\e \in \mathcal{D}, \ \ \ \ \nabla_\xi W_{N}(\xi_\e,\mathtt{d})\rightarrow 0.
\end{equation}
\end{theorem}

A family of solutions $u_\e$ of \eqref{equation}--\eqref{BC} with properties \eqref{1.21} and \eqref{1.22} in some set $\mathcal{D}$ compactly contained in $\{\xi \in \O^k; \xi_i \neq \xi_j, \ \text{ if } i\neq j \}$, is refereed as a $k$-vortex solution with degrees $\mathtt{d}$. We can deduce from the previous theorem the following result:

\begin{theorem}\label{theorem2}
For the semi-stiff problem \eqref{equation}--\eqref{BC}, for $\e$ small enough the following facts hold:
\begin{itemize}
\item[a)] A $1$-vortex solution with degree $1$ always exists. In particular for simply connected domains a solution of \eqref{equation}--\eqref{BC} in $\I_1$ always exists.
\item[b)] Two $2$-vortex solutions with degrees $(+1,-1)$, always exist. In particular for simply connected domain $\O$ there exist at least two solutions of \eqref{equation}--\eqref{BC} in $\I_0$.
\item[c)] Assume that $\O$ is not simply connected. Then, given any $m\geq 1$, there exists an $m$-vortex solution with degrees $(1,...,1)\in \mathbb{Z}^m$.
\end{itemize}
\end{theorem}

Point a) improves a previous result of Lamy-Mironescu \cite{LamyMironescu2014} on the existence of critical points of $E_\e$ in $\I_1$ in \emph{almost every} simply connected domain, for $\e$ small enough. Indeed they proved the existence of critical points of $E_\e$ with prescribed degrees under some non-degeneracy assumptions of the domain and then, they proved that in the case $d=1$ almost every domain satisfy these assumptions. Since we do not require any non-degeneracy assumption we are able to state the same result for \emph{every} simply connected domain. Point b) shows the existence of non-trivial solutions in $\I_0$  for simply connected domains. To our knowledge these solutions are new in the literature. Point c) provides us with the existence of new solutions with prescribed degrees in multiply connected domains. These solutions are different from the ones constructed by Berlyand-Rybalko and Dos Santos since their solutions have vortices at a distance $o(\e)\rightarrow 0$ of the boundary of the domain whereas our solutions have vortices well-inside the domain (at a distance of the boundary greater than $\delta>0$ for a fixed $\delta$ independent of $\e$).

The organization of the paper follows closely the one of \cite{delPinoKowalczykMusso2006}. We first compute the error estimate of our first approximation $w_{N\e}$. Then we give the ansatz, i.e.\ the form of the perturbation of the approximation under which we look for solutions of \eqref{equation}--\eqref{BC}. This allows us to give a new formulation of the problem. We then study a suitably projected problem (its linear and nonlinear version). From that study we deduce that critical points of $E_\e$ in $\I$ correspond to critical points of an approximation of the renormalized energy. At last we conclude the proof of Theorems \ref{theorem1} and \ref{theorem2} as in \cite{delPinoKowalczykMusso2006}.

\section{First approximation and error estimates}

Let us fix a number $k\geq1$, a $k$-tuple $\mathtt{d}\in\{-1,1\}^k$, a small number $\delta>0$ and $\xi \in \O_\delta^k$ with
\begin{equation}\label{defOmegadelta}
\O_\delta^k=\{ \xi \in \O^k, |\xi_i-\xi_j|>2\delta \text{ for all } i\neq j, \ \dist(\xi,\p\O)>2\delta \}.
\end{equation}
Let $I_{\pm}$ be the respective sets of indices associated to $\pm 1$ in $\mathtt{d}$. We rewrite our approximation $w_{N\e}$ given by \eqref{approximation}. For that we note that the solution $\varphi_N$ to the problem \eqref{varphiprblem}--\eqref{varphiproblem2} can be decomposed as
\begin{equation}\nonumber
\varphi_N(x)=\sum_{j=1}^k d_j\varphi_j^*(x),
\end{equation}
where
\begin{equation}\nonumber
\Delta \varphi_j^*=0 \ \text{ in } \O,
\end{equation}
\begin{equation}\label{varphiBC}
\p_\nu \varphi_j^*=-\frac{(x-\xi_j)^\perp\cdot \nu}{|x-\xi_j|^2} \ \text{ on } \p \O, \ \ \ \int_\O \varphi_j^*=0.
\end{equation}
We observe that if $\theta(x-\xi_j)$ denotes the polar argument around the point $\xi_j$ then we have precisely
\[ \nabla \theta_j(x)=\frac{(x-\xi_j)^\perp}{|x-\xi_j|^2}.\]
We can thus rewrite $w_{N\e}$ as
\[w_{N\e}(x)=U_0(x)\prod_{j\in I_+} e^{i(\theta_j(x)+\varphi_j^*(x))} \prod_{j\in I_-} e^{-i(\theta_j(x)+\varphi_j^*(x))},\]
with
\[ U_0(x)=\prod_{j\in I_+\cup I_-} U\left( \frac{|x-\xi_j|}{\e}\right). \]

We choose $w_{N\e}$ as a first approximation for a solution of the problem \eqref{equation}--\eqref{BC} (Note that this is the same approximation as the one used for the Neumann problem in \cite{delPinoKowalczykMusso2006}). We let $\O_\e=\e^{-1}\O$. For $u$ defined in $\O$ we write $v(y)=u(\e y)$ for $y\in \O_\e$. Then $u$ is a solution of \eqref{equation}--\eqref{BC} if and only if $v$ is a solution of
\begin{equation}\label{equationnewform}
\left\{
\begin{array}{rcll}
\Delta v+(1-|v|^2)v=0 & \text{ in } \O_\e, \\
|v|=1 & \text{ on } \p \O_\e, \\
v\wedge \p_\nu v=0 & \text{ on } \p \O_\e.
\end{array}
\right.
\end{equation}

In the sequel we let
\[V_0(y):=w_{N\e}(\e y), \ \ \ \ \xi_j':=\frac{\xi_j}{\e} \ \text{ and } \tilde{\varphi}_j^*(y):=\varphi_j^*(\e y). \]
We define the errors of $V_0$ with respect to the equations \eqref{equationnewform} as
\begin{equation}\label{error1}
E:=\Delta V_0+(1-|V_0|^2)V_0
\end{equation}
\begin{equation}\label{error2}
F:= V_0 \wedge \p_\nu V_0 \text{ on } \p \O_\e
\end{equation}
\begin{equation}\label{error3}
G:=\log |V_0| \text{ on } \p \O_\e.
\end{equation}
We then have
\begin{lemma}\label{Lemma2.1}
\begin{itemize}
\item[1)] There exists a constant $C$, depending on $\delta$ and $\O$ such that for all small $\e$ and all points $\xi \in \O_\delta^k$ we have
    \begin{equation}
    \sum_{j=1}^k \|E\|_{C^1(|y-\xi_j'|<3)} \leq C\e.
    \end{equation}
    Moreover, we have that $E=iV_0(y)[R_1+iR_2]$ with $R_1, R_2$ real-valued and
    \begin{equation}
    |R_1(y)|\leq C\e \sum_{j=1}^k \frac{1}{|y-\xi_j'|^3}, \ \ \ \ |R_2(y)|\leq C\e \sum_{j=1}^k \frac{1}{|y-\xi_j'|}
    \end{equation}
    if $|y-\xi_j'|>1$ for all $j$.
\item[2)] Besides we have
\begin{eqnarray}
F=V_0\wedge \p_\nu V_0 =0 \ \text{ on } \p \O_\e, \\
 \|G\|_\infty+\e^{-1}\|\nabla G\|_\infty+\e^{-2}\| D^2G \|_{L^\infty}=O(\e^2) \text{ on } \p \O_\e.
\end{eqnarray}

\end{itemize}
\end{lemma}

\begin{proof}
The first point is proved in \cite[Lemma 2.1]{delPinoKowalczykMusso2006}, we only prove the second point. We have that
\begin{eqnarray}\label{gradientV0}
\nabla V_0(y)&=&V_0(y)\left\{\sum_{j\in I_+\cup I_-} \frac{\nabla U(|y-\xi_j'|)}{U(|y-\xi_j'|)} \right. \\
& & \phantom{aaa} \left. + i\left[ \sum_{j\in I_+} (\nabla \theta_j(y)+\nabla \tilde{\varphi}_j(y))- \sum_{j\in I_-} (\nabla \theta_j(y)+\nabla \tilde{\varphi}_j(y)) \right] \right\},\nonumber
\end{eqnarray}
where with a slight abuse of notation we have called $\theta_j(y)=\theta(y-\xi_j')$. Thus
\begin{eqnarray}
\p_\nu V_0(y)&=&\nabla V_0(y)\cdot \nu \nonumber \\
&=& V_0(y) \left\{ \sum_{j\in I_+\cup I_-} \frac{\p_\nu U(|y-\xi_j'|)}{U(|y-\xi_j'|)} \right.\nonumber \\
& &\phantom{aaa} \left. + i \left[\sum_{j\in I_+} (\p_\nu \theta_j+\p_\nu \tilde{\varphi}^*) -\sum_{j\in I_-} (\p_\nu \theta_j+\p_\nu \tilde{\varphi}_j^*) \right] \right\} \nonumber
\end{eqnarray}
By using the boundary conditions \eqref{varphiBC} we find that $\p_\nu \theta_j+\p_\nu \tilde{\varphi}_j^*=0$ on $\p \O_\e$. Hence $\p_\nu V_0(y)=V_0(y) \sum_{j\in I_+ \cup I_-} \frac{\p_\nu U(|y-\xi_j'|)}{U(|y-\xi_j'|)}$. But since $U$ is real-valued it comes that
\[F= V_0 \wedge \p_\nu V_0 =0 \ \text{ on } \p \O_\e. \]

From the estimates for $U$ \eqref{estimeessurlesderivees} we find directly that $\|G\|_\infty= O(\e^2)$ on the boundary. Besides we have that
\[ \log |V_0| =\sum_{j=1}^k \log U( |y-\xi'_j|). \]
Hence  
\begin{equation}\nonumber
\nabla \left( \log |V_0| \right) =\sum_{j=1}^k \frac{\nabla U( |y-\xi'_j|)}{U(|y-\xi'_j|)}
\end{equation}
and for $k,l=1,2$ we have
\begin{equation}\nonumber
\p_{kl}^2 \left( \log |V_0| \right)=\sum_{j=1}^k\frac{ \p_{kl}^2 U(|y-\xi'_j|)}{U(|y-\xi'_j|}-\frac{\p_k U(|y-\xi'_j|) \p_l U (|y-\xi'_j|)}{U^2(|y-\xi'_j|}.
\end{equation}
Now we use the asymptotic behaviour of $U$ and its derivative given by \eqref{estimeessurlesderivees} and \eqref{estimeesurladeriveeseconde} to obtain that, since the vortices $\xi'_j$ are at a distance greater than $\frac{2\d}{\e}$ from the boundary we have
\begin{equation}
\| \nabla G \|_{L^\infty} =O(\e^3) \ \ \text{ and } \ \ \| D^2G\|_{L^\infty}=O(\e^4).
\end{equation}


\end{proof}

\section{Formulation of the problem}

We will look for a solution of the problem \eqref{equation}--\eqref{BC} in the form of a small perturbation of $V_0$. There are several ways to write such a perturbation and we will follow the approach of \cite{delPinoKowalczykMusso2006}. That means that we will use an ansatz which is additive near the vortices (i.e.\ of the form $V_0+\psi$ near the vortices) and which is multiplicative away from the vortices (of the form $V_0e^{i\psi}$).

Let $\tilde{\eta}:\R \rightarrow \R$ be a smooth cut-off function such that $\tilde{\eta}(s)=1$ for $s\leq 1$ and $\tilde{\eta}(s)=0$ for $s\geq 2$. We define
\[ \eta(y)= \sum_{j\in I_+ \cup I_-} \tilde{\eta}(|y-\xi_j'|). \]
We shall look for a solution of \eqref{equationnewform} of the form
\begin{equation}\label{ansatz}
v(y)=\eta(V_0+iV_0\psi)+(1-\eta)V_0 e^{i\psi},
\end{equation}
where $\psi$ is small, however, possibly unbounded near the vortices. We write $\psi=\psi_1+i\psi_2$ with $\psi_1$ and $\psi_2$ real-valued. This very same ansatz was used in similar contexts: in order to construct traveling wave for the Schr\"odinger map equation \cite{LinWei2010} and vortex rings for the Gross-Pitaevskii equation \cite{LinWeiYang2013}, \cite{Weiyang2012}. We set
\begin{equation}\label{defphi}
\Phi=iV_0\psi,
\end{equation}
and we require $\Phi$ to be bounded (and smooth) near the vortices. We observe that
\begin{equation}\nonumber
v=V_0+iV_0\psi+(1-\eta)V_0[e^{i\psi}-1-i\psi],
\end{equation}
We let
\begin{equation}\label{gamma1}
\gamma_1(y)=(1-\eta)V_0[e^{i\psi}-1-i\psi]
\end{equation}
function supported in the set $\{ y \in \O_\e; |y-\xi_j'|>1 \ \text{ for all } j \}$. 
We first derive the boundary conditions that $\psi$ must satisfy to be a solution of the two last 
equations of \eqref{equationnewform}. We remark that, since the vortices $\xi_j$ are at a distance 
greater than $2$ from the boundary (they are at a distance $\frac{2\d}{\e}$ from the boundary $\p \O_
\e$), we have $v=V_0e^{i\psi}$ near the boundary. That is $v=V_0e^{i\psi_1}e^{-\psi_2}$. We thus have 
that $|v|=1$ on $\p \O_\e$ if and only if
\begin{equation}\label{psi2BC}
\psi_2=\log |V_0| \text{  on } \p \O_\e .
\end{equation}
Furthermore we have
\begin{eqnarray}
\p_\nu v&=& \p_\nu V_0e^{i\psi}+iV_0\p_\nu \psi e^{i\psi} \nonumber \\
&=& (\p_\nu V_0+iV_0\p_\nu \psi_1-V_0\p_\nu \psi_2)e^{i\psi}. \nonumber
\end{eqnarray}
Thus, by using that $V_0 \wedge \p_\nu V_0=0$, we find that $v\wedge \p_\nu v=|V_0|^2\p_\nu \psi_1$ on $\p \O_\e$ and $v\wedge \p_\nu v=0$ if and only if
\begin{equation}\label{psi1BC}
\p_\nu \psi_1=0 \text{  on } \p \O_\e.
\end{equation}
We end up with a Dirichlet boundary condition for $\psi_2$ and with a homogeneous Neumann boundary condition for $\psi_1$: 

Now we find the equations for $\psi$ in $\O$. We write these equations in a different way if we are far away from the vortices or near the vortices. 

\textit{In the region $\{ y\in \O_\e; |y-\xi_j'|>2 \ \text{ for all } 1\leq j\leq k\}$}:

By using that $\psi=\psi_1+i\psi_2$ with $\psi_1, \psi_2$ real-valued and that  $v(y)=V_0(y)e^{i\psi}$ in that region we see that 
\begin{equation}
\Delta v +(1-|v|^2) v =  \left [ iV_0\mathcal{L}^\e(\psi) -E +V_0(\nabla \psi)^2-V_0|V_0|^2(1-e^{-2\psi_2}-2\psi_2) \right]e^{i\psi}
\end{equation}
with
\begin{equation}\label{operatorLe}
\L^\e(\psi) =\Delta \psi+2\frac{\nabla V_0}{V_0}\cdot \nabla \psi -2i|V_0|^2\psi_2+\eta \frac{E}{V_0}\psi
\end{equation}
and $E$ given by \eqref{error1}. Note that $\eta=0$ in that region but we write the operator $\L^\e$ like that to be consistent with the sequel. The function $v$ is solution of the GL equations in that region if and only if 
\begin{eqnarray}\label{pblineariseloindesvortex}
\mathcal{L}^\e(\psi)=R-i(\nabla \psi)^2+i|V_0|^2(1-e^{-2\psi_2}-2\psi_2), 
\end{eqnarray}
with
\begin{equation}
R=iV_0^{-1}E.
\end{equation}
Let us describe in a more accurate form the equation above. Let us fix an index $1\leq j \leq k$ and let us define $\alpha_j$ by the relation
\begin{equation}\label{defalpha}
V_0(y)=w(y-\xi_j')\alpha_j(y),
\end{equation}
where by $w$ we mean $w_+$ or $w_-$ depending whether $j\in I_+$ or $j\in I_-$, in other words
\begin{equation}\label{defalpha2}
\alpha_j(y)=e^{i\varphi_N(\e y)}\prod_{l\neq j} w(y-\xi_l').
\end{equation}
For $|y-\xi_j'|<\frac{\delta}{\e}$, since $\a_j$ does not vanish in that region, there are two real functions $A_j$ and $B_j$ so that
\begin{equation}\label{defalpha3}
\alpha_j(y)=e^{iA_j(y)+B_j(y)}.
\end{equation}
These functions are defined by
\begin{eqnarray}
A_j(y)&=&\varphi_N(\e y)+\sum_{l\neq j} \theta_l(y-\xi_l') \nonumber \\
B_j(y)&=& \sum_{l\neq j} \log U(|y-\xi_l'|). \nonumber
\end{eqnarray}
Furthermore, a direct computation shows that, in this region, one has
\begin{equation}\label{estimateonA}
\nabla A_j(y)=O(\e), \ \ \ \ \ \Delta A_j(y)=0
\end{equation}
and
\begin{equation}\label{estimateonB}
\nabla B_j(y)=O(\e^3), \ \ \ \ \ \Delta B_j(y)=O(\e^4).
\end{equation}
Observe that the estimates \eqref{estimateonA}, \eqref{estimateonB} hold true in any region of points at a distance greater than $\frac{\delta}{\e}$ from any $\xi_l'$, with $l\neq j$. \\ Then equation \eqref{pblineariseloindesvortex} in the region $\{ y\in \O_\e; |y-\xi_j'|>2 \ \text{ for all } 1\leq j\leq k\}$ becomes
\begin{eqnarray}\label{equationonpsi1}
\Delta \psi_1&+&2\left( \nabla B_j+ \frac{U'(|y-\xi_j'|)}{U(|y-\xi_j'|)}\frac{y-\xi_j'}{|y-\xi_j'|}\right) \cdot \nabla \psi_1 \nonumber \\
 & & \phantom{aaaaaaa}-2\left(\nabla A_j+\nabla \theta_j(y)\right)\cdot \nabla \psi_2+2\nabla \psi_1\cdot \nabla \psi_2-R_1=0, 
\end{eqnarray}
and
\begin{eqnarray}\label{equationonpsi2}
\Delta \psi_2-2|V_0|^2\psi_2 &+ &2\left( \nabla B_j+\frac{U'(|y-\xi_j'|)}{U(|y-\xi_j'|)}\frac{y-\xi_j'}{|y-\xi_j'|}\right) \cdot \nabla \psi_2 + 2\left( \nabla A_j+\nabla \theta_j(y) \right)\cdot \nabla \psi_1 \nonumber \\
& &\phantom{aa }+|V_0|^2(e^{-2\psi_2}-1-2\psi_2)+|\nabla \psi_1|^2-|\nabla \psi_2|^2-R_2=0.
\end{eqnarray}

\textit{In the region $ \O \setminus \{ y\in \O_\e; |y-\xi_j'|>2 \ \text{ for all } 1\leq j\leq k\}$:}
We have that $v$ is a solution of \eqref{equationnewform} if and only if
\begin{equation}\label{pblinearise}
\L^\e(\psi)= R+M(\psi) \ \text{ in } \O_\e,
\end{equation}
with $\L^\e$ and $R$ as before and $M(\psi)$ is the nonlinear operator defined by
\begin{eqnarray}\label{nonlinearoperatorN}
M(\psi)&=&iV_0^{-1}[ \Delta \gamma_1 +(1-|V_0|^2)\gamma_1-2\text{Re}(\overline{V_0}iV_0\psi)(iV_0\psi+\gamma_1) \nonumber \\
 & & -(2\text{Re}(\overline{V_0}\gamma_1)+|iV_0\psi+\gamma_1|^2)(V_0+iV_0\psi+\gamma_1)]+(\eta-1)\frac{E}{V_0}\psi,
\end{eqnarray}
where $\gamma_1$ is defined by \eqref{gamma1}. Besides in that region we have \[\Delta v +(1-|v|^2v)= iV_0 \left(\L^\e(\psi) -R -M(\psi) \right).\] 

Now we shall rewrite the first equation of problem \eqref{equationnewform}  in terms of the function $\phi$ defined in \eqref{defphi} by $ \phi=iV_0 \psi$ in an intermediate region. Let $1\leq j \leq k$ be fixed, in the region $\{ |y-\xi_j'| <\frac{\d}{\e} \}$ we introduce the translated variable $z:=y-\xi_j'$. We define the function $\phi_j(z)$ through the relation
\begin{equation}\label{defphi_j}
\phi_j(z)=iw(z)\psi(y), \ \ |z|<\frac{\delta}{\e},
\end{equation}
with $y=\xi_j'+z$ namely
\begin{equation}\nonumber
\phi(y)=\phi_j(z)\alpha_j(z),
\end{equation}
where, with abuse of notation, we write $\alpha_j(z)$ to mean the function $\alpha_j(y)$ defined in \eqref{defalpha}. Hence in the translated variable, the ansatz \eqref{ansatz} becomes in that region:
\begin{equation}\label{ansatznearvortices}
v(y)=\alpha_j(z)\left(w(z)+\phi_j(z)+(1-\tilde{\eta}(z))w(z)\left[ e^{\frac{\phi_j(z)}{w(z)}}-1-\frac{\phi_j(z)}{w(z)}\right] \right).
\end{equation}
Let us call $\gamma_j=(1-\tilde{\eta})w\left[e^{\frac{\phi_j}{w}}-1-\frac{\phi_j}{w}\right]$. The support of this function is contained in the set $\{|z|>1\}=\{|y-\xi_j'|>1\}$. Let us consider the operator $L_j^\e$ defined in the following way: for $\phi_j, \psi$ linked through formula \eqref{defphi_j} we set
\begin{equation}\label{defL_j}
L_j^\e(\phi_j)(z):=iw(z)\mathcal{L}^\e(\psi)(\xi_j'+z).
\end{equation}
Then another way to say that $v$ solves \eqref{equationnewform} in the region $\{|y-\xi_j'|<\frac{\delta}{\e}\}$ is
\begin{equation}\label{eqL_j0}
L_j^\e(\phi_j)=\tilde{R_j}+\tilde{N_j}(\phi_j),
\end{equation}
where explicitly $L_j^\e$ becomes (we use $\psi=\frac{\phi_j}{iw}$ and $V_0=\alpha_j w$ to see that)
\begin{eqnarray}\label{eqL_j}
L_j^\e(\phi_j)=L^0(\phi_j)+2(1-|\alpha_j|^2)\text{Re}(\overline{w}\phi_j)w+2\frac{\nabla \alpha_j}{\alpha_j}\cdot \nabla \phi_j \nonumber \\
+2i\phi_j \frac{\nabla \alpha_j}{\alpha_j}\cdot \frac{\nabla w}{w}+\tilde{\eta}\frac{E_j}{V_0^j}\phi_j,
\end{eqnarray}
where $L^0$ is the linear operator defined by
\begin{equation}\nonumber
L^0(\phi)=\Delta \phi+(1-|w|^2)\phi-2\text{Re}(\overline{w}\phi)w.
\end{equation}
The term $\tilde{R}_j$ in \eqref{eqL_j0} is
\begin{equation}\label{defRj}
\tilde{R}_j=-\frac{E_j}{\alpha_j},
\end{equation}
 with $E_j$ which is given by
\begin{equation}\label{defE_j}
E_j=\Delta V_0^j+(1-|V_0^j|^2)V_0^j,
\end{equation}
where $V_0^j$ is the function $V_0$ translated to $\xi_j'$, namely $V_0^j(z)=V_0(z+\xi_j')$. Observe that, in terms of $\alpha_j$, $E_j$ takes the expression
\begin{equation}\label{defE_j2}
E_j=2\nabla \alpha_j \cdot \nabla w +w\Delta \alpha_j+(1-|\alpha_j|^2)|w|^2\alpha_j w.
\end{equation}
The nonlinear term $\tilde{N}_j(\phi_j)$ is given by
\begin{eqnarray}\label{defNj}
\tilde{N}_j(\phi_j)=-\left[ \frac{\Delta (\alpha_j \gamma_j)}{\alpha_j}+(1-|V_0^j|^2)\gamma_j-2|\alpha_j|^2\text{Re}(\overline{w}\phi_j)(\phi_j+\gamma_j) \right. \nonumber \\
\left.\phantom{\frac{\Delta (\alpha_j \gamma_j)}{\alpha_j}} -(2|\alpha_j|^2\text{Re}(\overline{w}\gamma_j)+|\alpha_j|^2|\phi_j+\gamma_j|^2)(w+\phi_j+\gamma_j)\right]+(\tilde{\eta}-1)\frac{E_j}{V_0^j}\phi_j 
\end{eqnarray}
 for  $|z|<2$, and 
\begin{eqnarray}
\tilde{N}_j(\phi_j)= iw(z) \left[i|V_0|^2\left(1-e^{-2\IM(\frac{\phi_j}{w})}-2\IM(\frac{\phi_j}{w}) \right)-i(\nabla (\frac{\phi_j}{w}) )^2 \right] .
\end{eqnarray}
 for  $2<|z|<\frac{\d}{\e}$. Taking into account the explicit form of the function $\alpha_j$ we get
\begin{equation}\label{estimatesonalpha_j}
\nabla \alpha_j(z)=O(\e), \ \ \ \Delta \alpha_j(z)=O(\e^2), \ \ \ |\alpha_j(z)|=1+O(\e^2)
\end{equation}
provided that $|y-\xi_j'|<\frac{\delta}{\e}$. With this in mind, we see that the linear operator $L_j^\e$ is a small perturbation of $L^0$.

To sum up we are led to the following problem in $\psi$:
\begin{equation}\label{fullproblem}
\left\{
\begin{array}{rcll}
\L^\e(\psi)&=& R+N(\psi) & \text{ in } \O_\e, \\
\p_\nu \psi_1&=&0 &\text{  on } \p \O_\e, \\
\psi_2&=&\ln |V_0| & \text{  on } \p \O_\e .
\end{array}
\right. 
\end{equation}
with $N$ defined by 
\begin{equation}
N(\psi):= \begin{cases}
-i(\nabla \psi)^2+i|V_0|^2(1-e^{-2\psi_2}-2\psi_2) \ \text{ if } y \in \{ |y-\xi_j'|>2 \ \text{ for all } 1\leq j\leq k\}, \\
M(\psi) \ \text{ elsewhere }
\end{cases}
\end{equation}
and $M$ defined by \eqref{nonlinearoperatorN}.

We intend to solve the problem \eqref{fullproblem}. For that we would like to invert the operator $\L^\e$ in order to express this problem as a fixed point problem. However we do not expect the operator $\L^\e$ to be invertible (as in \cite{delPinoKowalczykMusso2006}). But working in an appropriated orthogonal to the kernel of $\L^\e$ we can invert that operator. In Sections $4-5$ we solve a projected version of the problem \eqref{fullproblem}. The next step is to adjust the vortices $\xi$ in order to obtain an actual solution of the problem \eqref{equationnewform} (variational reduction), this is done exactly as in \cite{delPinoKowalczykMusso2006}. The theorems are then a consequence of the analysis of Section $8$ of \cite{delPinoKowalczykMusso2006} to which we refer. 

\section{Projected linear theory for $\mathcal{L}^\e$}

Let us consider a small, fixed number $\delta>0$, and points $\xi \in \O_\delta^k$, the set defined in \eqref{defOmegadelta}. We also call $\xi_j'=\frac{\xi_j}{\e}$. For $c_0\in \R$, we consider the following linear problem:

\begin{equation}\label{linearequation1}
\mathcal{L}^\e(\psi)=h+c_0\e^2\chi_{\O_\e \setminus \cup_{j=1}^k B(\xi_j',\delta / \e)} \ \ \text{ in } \O_\e,
\end{equation}
\begin{equation}\label{linearBC1}
\p_\nu \psi_1=0 \ \ \ \ \text{ and } \psi_2=g \ \ \text{ on } \p \O_\e,
\end{equation}
\begin{equation}\label{orthorelation}
\int_{\O_\e \setminus \cup_{j=1}^k B(\xi_j',\delta / \e)} \psi_1 =0, \ \ \text{Re} \int_{|z|<1} \overline{\phi}_jw_{x_l}=0, \ \text{for all } j,l.
\end{equation}

The operator $\mathcal{L}^\e$ is given by \eqref{operatorLe}, $\psi=\psi_1+i\psi_2$ with $\psi_1, \psi_2$ real-valued and $\phi_j$ is the function defined from $\psi$ by the relation \eqref{defphi_j}: $\phi_j=iw(z)\psi(z)$.  For a set $A$, we denote by $\chi_A$ its characteristic function defined as
\begin{equation}\nonumber
\chi_A(y)=1 \ \text{ if } y\in A, \ \ \ \chi_A(y)=0, \ \text{ otherwise }.
\end{equation}

 Note that the constant $c_0$ along with the condition $\int_{\O_\e \setminus \cup_{j=1}^k B(\xi_j',\delta / \e)} \psi_1 =0$ are introduced in order to have existence and uniqueness for this problem. Indeed $\psi_1$ satisfies homogeneous Neumann boundary condition and this type of conditions appears in problems with Neumann boundary conditions. We will establish a priori estimates for this problem. To this end we shall conveniently introduce adapted norms.  Let us fix numbers $0< \sigma,\gamma<1$ and $p>2$, let us denote $z=y-\xi_j'$, and $r_j=|y-\xi_j'|=|z|$. We define
\begin{eqnarray}\label{norme*}
\|\psi\|_*= \sum_{j=1}^k \| \phi_j\|_{W^{2,p}(|z|<3)} +\sum_{j=1}^k \left[ \| \psi_1\|_{L^\infty(r_j>2)}+\|r_j\nabla \psi_1\|_{L^\infty(r_j>2)}\right] \nonumber \\
+ \sum_{j=1}^k \left[ \| r_j^{1+\sigma}\psi_2 \|_{L^\infty(r_j>2)}+\|r_j^{1+\sigma} \nabla \psi_2\|_{L^\infty(r_j>2)} \right],
\end{eqnarray}

\begin{equation}\label{norme**}
\|h\|_{**}=\sum_{j=1}^k \|\tilde{h}_j\|_{L^p(|z|<3)}+\sum_{j=1}^k\left[ \|r_j^{2+\sigma}h_1\|_{L^\infty(r_j>2)}+\|r_j^{1+\sigma}h_2\|_{L^\infty(r_j>2)} \right].
\end{equation}
Here we have denoted $\tilde{h}_j(z)=iw(z)h(z+\xi_j')$. Besides, we define
\begin{eqnarray}\label{norme***}
\|g\|_{***}=\e^{-1-\sigma}\|g\|_{L^\infty(\p \O_\e)}+\e^{-2-\sigma}\|\nabla g\|_{L^\infty(\p \O_\e)}+\e^{-2-\sigma -\gamma} [\nabla g]_{\gamma,\p \O_\e}
\end{eqnarray}
where $[\nabla g]_{\gamma,\p \O_\e}=\sup _{ x\neq y, x,y \in \p \O_\e} \frac{|\nabla_{\p \O_\e} g(y)-\nabla_{\p \O_\e}g(x)|}{|y-x|^\gamma}$.
We want to prove the following result.

\begin{lemma}\label{lemma4.1}
There exists a constant $C>0$, dependent on $\delta$ and $\O$, such that for all $\e$ sufficiently small,  all points $\xi \in \O_\delta^k$, any constant $c_0$ in $\mathbb{R}$ and any solution of problem \eqref{linearequation1}--\eqref{orthorelation} we have
\begin{equation}\label{aprioriestimate4.7}
\|\psi\|_*\leq C\left[ |\log \e|\|h\|_{**}+\|g\|_{***} \right].
\end{equation}
\end{lemma}

\begin{proof}
We argue by contradiction. We assume that there exist sequences $\e_n\rightarrow 0$, $c_n$, points $\xi_{nj}\rightarrow \xi_j^* \in \O$ with $\xi_j^*\neq \xi_i^*$ for all $i\neq j$, and functions $\psi^n, h_n, g_n$ which satisfy
\begin{equation}\nonumber
\mathcal{L}^{\e_n}(\psi^n)=h_n+c_n\e_n^2\chi_{\O_{\e_n}\setminus \cup_{j=1}^k B(\xi_{nj}',\delta/\e_n)} \ \text{ in } \O_{\e_n},
\end{equation}
\begin{equation}\nonumber
\p_\nu \psi_1^n=0 \ \ \ \ \text{ and } \psi_2^n=g_n \text{ on } \p \O_{\e_n},
\end{equation}
\begin{equation}\nonumber
\int_{\O_{\e_n}\setminus \cup_{j=1}^k B(\xi_{nj}',\delta/\e_n)} \psi_1^n=0, \ \ \ \text{Re} \int_{|z|<1} \overline{\phi}_j^n w_{x_l}=0, \ \forall j=1,...k, \ l=1,2,
\end{equation}
with
\begin{equation}\nonumber
\|\psi^n\|_*=1 \ \ \text{ and } |\log \e_n|\|h_n\|_{**}+\|g_n\|_{***} \rightarrow 0.
\end{equation}
We will show that $\|\psi^n\|_* \rightarrow 0$ and that will be a contradiction. \\

\textit{Estimates in the region $\{ |y-\xi_{nj}'|>\d/ \e_n, \text{ for all } 1\leq j \leq n \}$: }

As a first step we shall show that the sequence of numbers $c_n$ is bounded. We observe from \eqref{operatorLe} (see also \eqref{equationonpsi1}) and estimates \eqref{estimateonA}, \eqref{estimateonB}, that on the region $\{ |y-\xi_{nj}'|>\frac{\delta}{\e_n} \text{ for all } 1\leq j \leq k\}$ we have
\begin{equation}\nonumber
\RE(\mathcal{L}^{\e_n}(\psi^n))=\Delta \psi_1^n+O(\e_n^3)\nabla \psi_1^n+O(\e_n)\nabla \psi_2^n,
\end{equation}
and hence, integrating on $\O_{\e_n} \setminus \cup_{j=1}^kB(\xi_{nj}',\delta/\e_n)$ we obtain that
\begin{eqnarray}
|c_n|\leq C\int_{\O_{\e_n} \setminus \cup_{j=1}^kB(\xi_{nj}',\delta/\e_n)} |h_n^1|+C\left|\int_{\cup_{j=1}^k \p B(\xi_{nj}',\delta/\e_n)} \p_\nu \psi_1^n -\int_{\p \O_\e} \p_\nu \psi_1^n \right| \nonumber \\
+\int_{\O_{\e_n} \setminus \cup_{j=1}^kB(\xi_{nj}',\delta/\e_n)}\left(O(\e_n^3)|\nabla \psi_1^n|+ O(\e_n)|\nabla \psi_2^n|\right). \nonumber
\end{eqnarray}
But for all $1\leq j\leq k$ we can use that, since we work in the region $\{r_j>\frac{\delta}{\e_n}\}$, for $n$ large enough we have
\[|\nabla \psi_1^n|\leq \| \nabla \psi_1^n\|_{L^\infty(r_j>2)}\leq \frac{1}{r_j}\|\psi^n\|_* \leq C\e_n\|\psi_n\|_*\]
and other similar estimate hold  for $\nabla \psi_2^n$ and $h_n$ so that we get (by using that $\p_\nu \psi_1^n=$ on $\p \O_\e$)
\begin{equation}\nonumber
|c_n|\leq C  + C \e_n^\sigma(\|h_n\|_{**}+\|\psi^n\|_{*}).
\end{equation}
It follows that $c_n$ is bounded. We assume then that $c_n\rightarrow c_*$ for some $c_*$ in  $\R$. We will show that $c_*=0$ and that $\psi^n$ converges to $0$. We set $\tilde{\psi}^n(x)=\psi^n(\frac{x}{\e_n})$. From the bounds assumed we have that for any number $\delta'>0$
\begin{equation}\nonumber
\Delta \tilde{\psi}_1^n=O(\e_n^\sigma)+c_n \chi_{\O \setminus \cup_{j=1}^k B(\xi_{nj},\d)} \ \text{ in } \O \setminus \bigcup_{j=1}^k B(\xi_{nj},\delta'),
\end{equation}
\begin{equation}\nonumber
\p_\nu \tilde{\psi}_1^n=0 \ \text{ on } \p \O,
\end{equation}
and, moreover,
\begin{equation}\nonumber
\|\tilde{\psi}_1^n\|_\infty \leq 1, \ \ \ \| \nabla \tilde{\psi}_1^n\|_\infty \leq C_{\delta'}.
\end{equation}
Passing to a subsequence, we then get that $\tilde{\psi}_1^n$ converges uniformly over compact subsets of $\O \setminus \{ \xi_1^*,...,\xi_k^*\}$ to a function $\tilde{\psi}_1^*$ with $|\tilde{\psi}_1^*|\leq 1$ which solves
\begin{equation}\nonumber
\Delta \tilde{\psi}_1^*=c_*\chi_{\O \setminus \cup_{j=1}^k B(\xi_j^*,\delta)} \ \text{ in } \O \setminus \{ \xi_1^*,..., \xi_k^*\},
\end{equation}
By integrating the previous relations we find that $c_*=0$ and hence that $\tilde{\psi}_1^*$ is a constant ($\tilde{\psi}_1^*$ can not contain a logarithmic part since $\|\tilde{\psi}_1^*\|_\infty \leq 1$). But passing to the limit in the orthogonality condition for $\tilde{\psi}_1^n$ provides $\int_{\O\setminus \cup_{j=1}^k B(\xi_j^*,\delta)} \tilde{\psi}_1^*=0$ and thus the constant must be zero. We conclude that $\tilde{\psi_1}^n$ goes to zero uniformly and in $C^1$-sense away from the points $\xi_1^*,..., \xi_k^*$. This implies in particular that
\begin{equation}\label{estimatepsi1loindesvortex}
|\psi_1^n|+\frac{1}{\e_n}|\nabla \psi_1^n|\rightarrow 0 \text{ on } \{ |z-\xi_{nj}'|\geq \frac{\delta}{2\e_n}, \ \forall \ 1\leq j \leq k\},
\end{equation}
uniformly. Note that we obtained an estimate for $\psi_1$ in a region a little bit larger than  $\{ |z-\xi_{nj}'|> \frac{\delta}{\e_n}, \ \forall \ 1\leq j \leq k\}$ because we will need it in the sequel. \\
Let us now consider the imaginary part of the equation. Using the definition of the operator $\mathcal{L}^\e$, cf.\ \eqref{operatorLe}, along with the estimates on $V_0$, $\nabla V_0$ and the bounds on $\| \psi^n\|_*$, $\|h_n\|_{**}$ and $\|g\|_{***}$ we find that
\begin{equation}\nonumber
-\Delta \psi_2^n+2\psi_2^n=o(\e_n^{1+\sigma}) \text{ in }  \O_{\e_n} \setminus \bigcup_{j=1}^k B(\xi_{nj}',\frac{\delta}{2\e_n}),
\end{equation}
\begin{equation}\nonumber
\psi_2^n=o(\e_n^{1+\sigma}) \ \text{ on } \p \O_{\e_n}.
\end{equation}
To be more precise we have \[\| -\Delta \psi_2^n+2\psi_2^n \|_{L^\infty}=o(\e_n^{1+\sigma}) \text{ in } \O_{\e_n} \setminus \bigcup_{j=1}^k B(\xi_j',\frac{\d}{2\e_n}) \] and 
\[ \| \psi_2^n \|_{C^{1,\gamma}}=o(\e_n^{1+\sigma}) \text{ on } \p \O_{\e_n}.\]   
Note also that we have 
\[ \psi_2^n =O(\e_n) \ \text{ on } \bigcup_{j=1}^k \p B(\xi_j',\frac{\d}{2\e_n}). \]
We first construct a barrier which proves that 
\begin{equation}
\| \psi_2 \|_\infty =o(\e_n^{1+\s}) \ \ \text{ in } \{ |z-\xi_{nj}'|\geq \frac{3\delta}{4\e_n}, \ \forall \ 1\leq j \leq k\}.
\end{equation}We let $\eta_1$ be defined by 
\[ \eta_1(y)= \sum_{j=1}^k \tilde{\eta}\left( \frac{\e_n |y-\xi_j'|}{\d}\right), \]
where $\tilde{\eta}$ is a smooth cut-off function such that $\tilde{\eta}(s)=1$ for $s\leq 1$ and $\tilde{\eta}(s)=0$ for $s\geq 2$. We take
\begin{equation}
\mathcal{B}_1=o(\e_n^{1+\s})[1-\eta_1(y)]+ \left[O(\e_n^{1+\s})e^{-r_j}+o(\e_n^{1+\s}) \right]\eta_1(y).
\end{equation}
We can check that $-\Delta \B_1 +2\B_1 \geq o(\e_n^{1+\s})$ in $\{ |z-\xi_{nj}'|\geq \frac{\delta}{2\e_n}, \ \forall \ 1\leq j \leq k\}$ and $ \B_1 =O(\e_n^{1+\s})$ on $\p \left (\O_{\e_n} \setminus \bigcup_{j=1}^k B(\xi_j',\frac{\d}{2\e_n}) \right)$. Thus 
\[ |\psi_2| \leq \B_1 =o(\e_n^{1+\s}) \text{ in } \{ |z-\xi_{nj}'|\geq \frac{3\delta}{4\e_n}, \ \forall \ 1\leq j \leq k\}.\] Now to estimate the $L^\infty$ norm of the gradient we use a combination of elliptic estimates (cf.\ Propositions \ref{EE1} and \ref{EEE2}) and this yields
\begin{equation}\label{estimateloindesvortex2}
\frac{1}{\e_n^{1+\sigma}}\left( |\psi_2^n|+|\nabla \psi_2^n|\right) \rightarrow 0 \text{ on } \{ |z-\xi_{nj}'|\geq \frac{\delta}{\e_n}, \ \forall \ 1\leq j \leq k\}.
\end{equation}

\textit{ Estimates in the region $\O_{\e_n} \setminus \{ |y-\xi_{nj}'|>\d/ \e_n, \text{ for all } 1\leq j \leq n \} $:} \\
Now we want to derive estimates on $\psi^n$ near the vortices. Let us consider a smooth cut-off function $\hat{\eta}$ with
\begin{equation}\nonumber
\hat{\eta}=\begin{cases}
1 \text{ if } s< \frac 12, \\
0 \text{ if } s>1,
\end{cases}
\end{equation}
and define
\begin{equation}\nonumber
\hat{\psi}^n(z)=\hat{\eta}\left( \frac{\e_n |z-\xi_{nj}'|}{\delta} \right)\psi^n(z).
\end{equation}

Let us compute the equation satisfied by $\hat{\psi}^n$. First observe that the derivative of $z \mapsto \hat{\eta}\left( \frac{\e_n |z-\xi_{nj}'|}{\delta} \right)$ are supported in the annulus $ \frac {\delta}{2\e_n}< |z-\xi_{nj}'|<\frac{\delta}{\e_n}$. In that region we have $|\psi_1^n|+\frac{1}{\e_n}|\nabla \psi_1^n|\rightarrow 0$ and $\left( | \psi_2^n| +|\nabla \psi_2^n|\right)=O(\e_n^{1+\s}) $. Thus, for real and imaginary parts we obtain the estimates


\begin{equation}\nonumber
\nabla_z\eta \left( \frac{\e_n |z-\xi_{nj}'|}{\delta} \right) \nabla \psi^n= \begin{pmatrix}
o(\e_n^2) \\
O(\e_n^{2+\sigma})
\end{pmatrix},
\ \ \ \ \psi^n\Delta_z \hat{\eta}\left( \frac{\e_n |z-\xi_{nj}'|}{\delta} \right)= \begin{pmatrix}
o(\e_n^2) \\
O(\e_n^{2+\sigma})
\end{pmatrix}.
\end{equation}
Furthermore, by computing $\frac{\nabla V_0}{V_0}$ (cf.\ formula  \eqref{gradientV0}) and by using appropriates estimates we also find that
\begin{equation}\nonumber
\psi^n\frac{\nabla V_0}{V_0}\nabla_z \hat{\eta}\left( \frac{\e_n |z-\xi_{nj}'|}{\delta} \right)= \begin{pmatrix}
o(\e_n^2) \\
O(\e_n^{3+\sigma})
\end{pmatrix}.
\end{equation}
Thus we get
\begin{equation}\nonumber
\mathcal{L}_n^\e(\hat{\psi}^n)=o(1) \begin{pmatrix}
\frac{1}{(|\log \e_n|)r_j^{2+\sigma}}+\e_n^2 \\
\frac{1}{|\log \e_n| r_j^{1+\sigma}}
\end{pmatrix} \text{ in } 2<| y-\xi_{nj}'|<\frac{\d}{\e_n},
\end{equation}
\begin{equation}\nonumber
\| iw\mathcal{L}_n^\e(\hat{\psi}^n)\|_{L^p(\{ |z|\leq 2\})}=o\left( \frac{1}{|\log \e_n|} \right) 
\end{equation}
and
\begin{equation}\label{4.8}
\hat{\psi}^n=0 \ \text{ on } \p B(\xi_{nj}',\frac{\delta}{\e_n}).
\end{equation}
The following intermediate result provides an outer estimate. In order to simplify the notations we omit the subscript $n$ in the quantities involved after the statement of this lemma:.
\begin{lemma}\label{outerestimates}
There exist positive numbers $R_0$, $C$ such that for all large $n$
\begin{multline}
\| \hat{\psi}^n_1\|_{L^\infty(r_j>R_0)}+\| r_j \nabla \hat{\psi}^n_1\|_{L^\infty(r_j>R_0)}+ \|r_j^{1+\sigma}\hat{\psi}^n_2\|_{L^\infty(r_j>R_0)}+ \|r_j^{1+\sigma} \nabla \hat{\psi}^n_2\|_{L^\infty(r_j>R_0)} \\
\leq C\left( \| \hat{\phi}^n\|_{C^1(r_j<R_0)}+o(1) \right),
\end{multline}
where $\hat{\phi}^n=iV_0\hat{\psi}^n$.
\end{lemma}
\begin{proof}
 Using \eqref{equationonpsi1} and \eqref{equationonpsi2} we obtain the following relations for $2<r_j<\frac{\d}{\e}$:
\begin{equation}\label{4.9}
-\Delta \hat{\psi}_1=O\left(\frac{1}{r_j^3}\right) \nabla \hat{\psi}_1+O\left( \frac{1}{r_j}  \right) \nabla \hat{\psi}_2 +o(1)\left( \frac{1}{r_j^{2+\sigma}}+\e^2 \right),
\end{equation}
\begin{equation}\label{4.10}
-\Delta \hat{\psi}_2+2|\alpha_j|^2|w_j|^2\hat{\psi}_2+O\left(\frac{1}{r_j^3}\right)\nabla \hat{\psi}_2=O \left( \frac{1}{r_j}\right) \nabla \hat{\psi}_1 +o(1)\frac{1}{r_j^{1+\sigma}},
\end{equation}
where $\alpha_j$ is given by \eqref{defalpha} and $w_j(y)=w(y-\xi_j')$. Note that $|\a_j|^2|w_j|^2\geq c>0$ for some $c>0$ in the region $r_j>2$ (cf.\ estimates \eqref{estimatesonalpha_j}). Let us call $p_1,p_2$ the respective right-hand sides of equations \eqref{4.9} and \eqref{4.10}. Then we see that
\begin{equation}\nonumber
|p_2|\leq \frac{B}{r_j^{1+\sigma}}, \ \ B=\| r_j^{\sigma} \nabla \hat{\psi}_1\|_{L^\infty(r_j>2)}+o(1).
\end{equation}
The function 
\begin{equation}
\mathcal{B}_2:= C \frac{B +\|\hat{\psi}_2\|{L^\infty(r_j=2})}{r_j^{1+\s}}, 
\end{equation}
satisfies $-\Delta \B_2 +2c \B_2 \geq \frac{B +\|\hat{\psi}_2\|{L^\infty(r_j=2})}{r_j^{1+\s}}$ in $\{ 2<r_j< \frac{\d}{\e}\}$ and $ \B_2 \leq |\hat{\psi}_2|$ on $\{ r_j=2\} \cup \{r_j=\frac{\d}{\e}\}$, for $C$ large enough, thus we obtain 
\begin{equation}\nonumber
|\hat{\psi}_2| \leq C \frac{B+\| \hat{\psi}_2\|_{L^\infty(r_j=2)}}
{r_j^{1+\sigma}}, \ \ \ \ 2< r_j< \frac{\delta}{\e}.
\end{equation}
To estimate the gradient of $\hat{\psi}_2$ we use a combination of elliptic estimates (cf.\ Proposition \ref{EE1} and \ref{EE3}) and we obtain 
\begin{equation}\nonumber
|\nabla \hat{ \psi}_2|+ |\hat{\psi}_2| \leq C\frac{B+\| \hat{\psi}_2\|_{L^\infty(r_j=2)}}
{r_j^{1+\sigma}}, \ \ \ \ 2< r_j< \frac{\delta}{\e}.
\end{equation}
Let us use these estimates to now estimate $p_1$. We have that
\begin{equation}\nonumber
|p_1|\leq \frac{C}{r_j^{2+\sigma}}\left( \|r_j^\sigma \nabla \hat{\psi}_1\|_{L^\infty(r_j>2)}+\|r_j^{1+\sigma} \nabla \hat{\psi}_2\|_{L^\infty(r_j>2)}+o(1) \right)+o(\e^2),
\end{equation}
and hence
\begin{equation}\nonumber
|p_1|\leq \frac{B'}{r_j^{2+\sigma}}+o(\e^2), \ \ B'=C\left( \|r_j^\sigma \nabla \hat{\psi}_1\|_{L^\infty(r_j>2)}+\| \hat{\psi}_2\|_{L^\infty(r_j=2)}+o(1)\right).
\end{equation}
We can see that the function 
\begin{equation}\nonumber
\B_3=\frac{B'}{\sigma^2}\left( 1-\frac{1}{r_j^\sigma}\right) +o(1)(\delta^2-r_j^2\e^2)+\|\hat{\psi}_1\|_{L^\infty(r_j=1)}
\end{equation}
satisfies $-\Delta \B_3 = \frac{B'}{r_j^{2+\sigma}}+o(\e^2)$ in $\{2<r_j<\frac{\d}{\e}\}$ and $\B_3\geq |\hat{\psi}_1|$ on $\{r_j=2\}\cup\{r_j=\frac{\d}{\e}\}$. Thus $|\hat{\psi}_1| \leq \B_3$  and 
\begin{equation}\nonumber
\| \hat{\psi}_1\|_{L^\infty(r_j>2)}\leq C+\| \hat{\psi}_1\|_{L^\infty(r_j=2)}.
\end{equation}
Now we seek for an estimate for $\nabla \hat{\psi}_1$. Let us define $\tilde{\psi}_1(z)=\hat{\psi}_1(\xi_j'+R(e+z))$ where $|e|=1$ and $R<\frac{\delta}{\e}$. Then for $|z|\leq \frac{1}{2}$ we have
\begin{equation}\nonumber
|\Delta \tilde{\psi}_1(z)| \leq CB'+o(1).
\end{equation}
Since also $|\tilde{\psi}_1|\leq CB'$ in this region, it follows from elliptic estimates that $|\nabla \tilde{\psi}_1(0)|\leq CB'$. Since $R$ and $e$ are arbitrary, what we have established is
\begin{equation}\nonumber
|\hat{\psi}_1|+|r_j\nabla \hat{\psi}_1| \leq C \left[ \|r_j^\sigma \nabla \hat{\psi}_1\|_{L^\infty(r_j>2)}+ \|\hat{\psi}_2\|_{L^\infty(r_j=2)}+o(1) \right]
\end{equation}
Now,
\begin{equation}\nonumber
\|r_j^\sigma \nabla \hat{\psi}_1\|_{L^\infty(r_j>2)} \leq R_0^\sigma\| \nabla \hat{\psi}_1\|_{L^\infty(2<r_j<R_0)}+\frac{1}{R_0^{1-\sigma}}\|r_j\nabla \hat{\psi}_1\|_{L^\infty(r_j>R_0)},
\end{equation}
thus fixing $R_0$ sufficiently large we obtain
\begin{equation}\nonumber
|\hat{\psi}_1|+|r_j\nabla \hat{}\psi_1| \leq C\left[ \|\nabla \hat{\psi}\|_{L^\infty(2<r_j<R_0)}+o(1) \right] \ \text{ for } r_j>R_0,
\end{equation}
and also
\begin{equation}\nonumber
|\hat{\psi}_2|+ |\nabla \hat{\psi}_2| \leq \frac{C}{r_j^{1+\sigma}}\left[ \| \nabla \hat{\psi}\|_{L^\infty(2<r_j<R_0)}+o(1)\right] \ \text{ for } r_j>R_0.
\end{equation}
Now let us define $\hat{\phi}$ through $\hat{\psi}=\frac{\hat{\phi}}{iV_0}$. We have $\| \nabla \hat{\psi} \|_{L^\infty} \leq C \|\hat{\phi}\|_{C^1}$ with $C$ which depends on the $C^1$ norm of $V_0$ (we note that this norm is finite in $2<r_j<R_0$). This concludes the proof of Lemma \ref{outerestimates}.
\end{proof}

\textit{Continuation of the proof of Lemma 4.1.} Let us go back to the contradiction argument. Since $\| \psi\|_*=1$ and since from \eqref{estimatepsi1loindesvortex} and \eqref{estimateloindesvortex2} the corresponding portion of this norm of $\psi$ goes to zero on the region $r_j>\frac{\delta'}{\e}$ for all $j$ for any given $\delta'$ (recall that $r_j\leq \frac{\text{diam}( \O)}{\e}$), we conclude by using the definition of the norm $\| \cdot \|_*$ and the previous lemma that there exists some $1\leq j\leq k$ and $m>0$ such that, for $R_0$ as in the lemma
\begin{equation}\label{4.11}
\| \hat{\phi}_j\|_{W^{2,p}(|z|<R_0)} \geq m,
\end{equation}
where, as in \eqref{defphi_j}
\begin{equation}\nonumber
\hat{\phi}_j(z)=iw(z)\hat{\psi}(\xi_j'+z), 
\end{equation}
and where we also used the Sobolev injections $W^{2,p} \hookrightarrow C^1$.
Let us consider the decomposition
\begin{equation}\nonumber
\hat{\psi}(\xi_j'+z)= \hat{\psi}^0(r)+\hat{\psi}^\perp(z), \ \ \ r=|z|,
\end{equation}
\begin{equation}\nonumber
\hat{\psi}^0(r)=\frac{1}{2\pi r}\int_{|z|=r} \hat{\psi}(\xi_j'+z)d\sigma(z),
\end{equation}
and correspondingly write
\begin{equation}\label{4.12}
\hat{\phi}_j=\hat{\phi}^0+\hat{\phi}^\perp, \ \ \ \hat{\phi}^0=iw\hat{\psi}^0, \ \ \ \hat{\phi}^\perp=iw\hat{\psi}^\perp.
\end{equation}
Using equations \eqref{4.8}, formula \eqref{eqL_j} and the estimates on $\alpha_j$ \eqref{estimatesonalpha_j} along with the fact that $\|\psi\|_*=1$ we can see that
\begin{equation}\nonumber
\left\{
\begin{array}{rcll}
L^0(\hat{\phi}_j)&=&G \ \text{ in } B(0,\frac{\delta}{\e}), \\
\hat{\phi}_j&=&0\ \text{ on } \p B(0,\frac{\delta}{\e}),
\end{array}
\right.
\end{equation}
where 
\begin{equation}\label{2222}
\|G\|_{L^p(\{|z|\leq 2\})}=o\left(\frac{1}{|\log \e|}\right) 
\end{equation}
and
\begin{equation}\label{1111}
H=-iw^{-1}G=o(1) \begin{pmatrix}
\frac{1}{|\log \e| r^{2+\sigma}}+\e^2 \\
\frac{1}{|\log \e| r^{1+\sigma}}
\end{pmatrix}
\text{ for } r>2.
\end{equation}
We also set
\begin{eqnarray}
H(\xi_j'+z)=H^0(r)+H^\perp(z), \ \ r=|z|, \nonumber \\
H^0(r)=\frac{1}{2\pi}\int_{|z|=r} H(\xi_j'+z)d\sigma(z), \nonumber
\end{eqnarray}
and we decompose $G=G^0+G^\perp$ in analogous way to \eqref{4.12}. We can then check that
\begin{equation}\nonumber
\left\{
\begin{array}{rcll}
L^0(\hat{\phi}^\perp)=G^\perp \ \ \text { in } B(0,\frac{\delta}{\e}), \\
\hat{\phi}^\perp =0 \ \ \text{ on } \p B(0,\frac{\delta}{\e}).
\end{array}
\right.
\end{equation}
From estimates \eqref{2222} and \eqref{1111}, by using H\"older's inequality and the fact that $\|\hat{\psi}\|_*$ is uniformly bounded we find that
\begin{equation}\nonumber
\RE \int_{B(0,\frac{\delta}{\e})} \overline{G}^\perp \hat{\phi}^\perp=o(1).
\end{equation}
Define $B(\phi,\phi)=-\RE \int_{B(0,\frac{\delta}{\e})} L^0(\phi) \overline{\phi}$. From the result in Lemma A.1 in \cite{delPinoKowalczykMusso2006}, it follows that there exists a number $\alpha>0$ such that
\begin{equation}\label{4.14}
\alpha \left\{ \int_{B(0,\frac{\delta}{\e})} \frac{|\phi^\perp|^2}{1+r^2}+\int_{B(0,\frac{\delta}{\e})}|\RE(\phi^\perp w)|^2+ \int_{B(0,\frac{\delta}{\e})}|\nabla \phi^\perp|^2 \right\} \leq B(\phi^\perp, \phi^\perp),
\end{equation}
where the orthogonality conditions
\begin{equation}\nonumber
\RE \int_{B(0,1)} \phi^\perp \overline{w}_{x_j}=0, \ \ j=1,2,
\end{equation}
are used (note that $\int_{B(0,1)} \hat{\phi}^0 \overline{w}_{x_j}=0$ thanks to Fubini's theorem). Now, since $B(\hat{\phi}^\perp, \hat{\phi}^\perp)=o(1)$, it follows that
\begin{equation}\nonumber
\int_{B(0,2R_0)} |\hat{\phi}^\perp|^2=o(1)
\end{equation}
and elliptic estimates yield $\hat{\phi}^\perp \rightarrow 0$ in $W^{2,p}$-sense in $B(0,R_0)$. Let us now consider $\hat{\phi}^0=iw\hat{\psi}^0$. Then from the equation $L^0(\hat{\phi}^0)=G^0$ we obtain
\begin{equation}\nonumber
\left\{
\begin{array}{rcll}
\Delta \hat{\psi}^0+2\frac{\nabla w}{w}\nabla \hat{\psi}^0-2i|w|^2\hat{\psi}_2^0&=&H^0 \ \text{ in } B(0,\frac{\delta}{\e}), \\
\hat{\psi}^0&=&0  \ \text{ on } \p B(0,\frac{\delta}{\e}).
\end{array}
\right.
\end{equation}
This equation translates into the uncoupled system
\begin{equation}\nonumber
\Delta \hat{\psi}_1^0+\frac{2U'}{U}\frac{\p \hat{\psi}_1^0}{\p r}=H_1^0(r),
\end{equation}
\begin{equation}\nonumber
\Delta \hat{\psi}_2^0+\frac{2U'}{U}\frac{\p \hat{\psi}_2^0}{\p r}-2U^2 \hat{\psi}_2^0=H_2^0(r)
\end{equation}
for $0<r<\frac{\delta}{\e}$. The first equation, plus the boundary condition has the unique solution
\begin{equation}\label{formulaexplicitpsi1}
\hat{\psi}_1^0(r)=-\int_r^\frac{\delta}{\e}\frac{ds}{sU^2(s)}\int_0^sH_1^0(t)U^2(t)tdt.
\end{equation}
Since
\begin{equation}\nonumber
H_1^0(r)=\frac{o(1)}{|\log \e|r^{2+\sigma}}+o(\e^2) \ \ r>2,
\end{equation}
and
\begin{equation}\nonumber
H_1^0(r)=o(\frac{1}{|\log \e|})\frac{1}{r} \ \ r<2
\end{equation}
(the last equality is true because $U(r)=O(r)$ for $r<1$) it follows from the formula above that $\hat{\psi}_1^0(r)=o(1)$. On the other hand, a barrier (of the form $o(1)(r)$) shows that on $\hat{\psi}_2^0$ we have the estimate $\hat{\psi}_2^0(r)=o(1)r$. As a conclusion we finally derive
\begin{equation}\nonumber
\int_{B(0,2R_0)} |\hat{\phi}^0|^2+|\hat{\phi}^\perp|^2=o(1),
\end{equation}
and hence, from elliptic estimates, $\hat{\phi}_j \rightarrow 0$ in a $W^{2,p}$-sense on $B(0,R_0)$. This is a contradiction with \eqref{4.11}. We obtain that $\|\psi\|_* \rightarrow 0$ and this is a contradiction with $\|\psi\|_*=1$. The lemma is proven.
\end{proof}

We now consider  the following projected linear problem.
\begin{equation}\label{4.15}
\mathcal{L}^\e(\psi)=h+c_0\e^2\chi_{\O_\e \setminus \cup_{j=1}^k B(\xi_j',\frac{\delta}{\e})}+ \sum_{j,l} c_{jl}\frac{1}{iV_0\overline{\a_j}}w_{x_l}(y-\xi_j')\chi_{\{r_j<1\}} \ \text{ in } \O_\e,
\end{equation}
\begin{equation}\label{4.16}
\p_\nu \psi_1=0, \ \ \ \ \psi_2=g \ \text{ on } \p \O_\e,
\end{equation}
\begin{equation}\label{4.17}
\int_{\O_\e \setminus \cup_{j=1}^k B(\xi_j',\frac{\delta}{\e})} \psi_1=0, \ \ \ \RE \int_{|z|<1} \overline{\phi}_jw_{x_l}=0, \ \ \forall j=1,...,k, \ l=1,2
\end{equation}
with
\begin{equation}\nonumber
\phi_j(z)=iw(z) \psi(\xi_j'+z).
\end{equation}

Here we have called (with some abuse of notation) $w(z)=w^{\pm}(z)$ if $j\in I_{\pm}$. The following is the main result of this section.

\begin{proposition}\label{Prop4.1}
There exists a constant $C>0$, dependent on $\delta$ and $\O$ but independent of $c_0$, such that for all small $\e$ the following holds: if $\|h\|_{**}+\|g\|_{***}<+\infty$ then there exists a unique solution $\psi=T_\e(h,g)$ to problem \eqref{4.15}--\eqref{4.17}. Besides,
\begin{equation}\label{4.18}
\|T_\e(h,g)\|_* \leq C\left[ |\log \e| \|h\|_{**}+\|g\|_{***} \right].
\end{equation}
Moreover, the constants $c_{lj}$ admit the asymptotic expression
\begin{equation}\label{4.19}
c_{lj}=-\frac{1}{c_*}\RE \int_{\{ |z|<\delta/\e \}} \tilde{h}_j\overline{w}_{x_l}+O(\e \log \e)\|\psi\|_*+O(\e^2) \|h\|_{**},
\end{equation}
where $c_*=\int_{B(0,1)}|w_{x_m}|^2$ for $m=1,2$. Here
\begin{equation}\nonumber
\tilde{h}_j(z)=iw(z) h(\xi_j'+z).
\end{equation}
\end{proposition}

\begin{proof}
First we remark that for the existence part we can always assume that $g=0$ (up to modification of the function $h$ in the right hand side). We denote by $(R.H.S)$ the right hand side of Equation \eqref{4.15}. We express the problem in terms of $\phi=iV_0\psi$.
Equation \eqref{4.15} can be written as
\begin{equation}\nonumber
\Delta \phi +(\eta -1)\frac{E}{V_0}\phi+(1-|V_0|^2)\phi-2\RE(\overline{V_0}\phi)V_0=(R.H.S)iV_0 \ \ \ \text{ in } \O_\e.
\end{equation}
Let us set
\begin{eqnarray}\label{spaceH}
H:= \left\{  \phi=iV_0 \psi \in H^1(\O_\e); \ \psi_2=0 \ \text{ on } 
\p \O_\e, \phantom{\int_{\O_\e \setminus \cup_{j=1}^k B(\xi_j',\frac{\delta}{\e})} }\right.  \nonumber \\ \phantom{aaaaaaaaa} \left. \int_{\O_\e \setminus \cup_{j=1}^k B(\xi_j',\frac{\delta}{\e})} \psi_1=0 \ \text{ and } \RE \int_{\{|z|<1\}} \overline{\phi}_jw_{x_l}=0 \ \forall \ j,l \right\},
\end{eqnarray}
and $[ \phi, \varphi]:=\RE \int_{\O_\e} \nabla \phi \nabla \overline{\varphi}$. The space $H$ endowed with the inner product $[\cdot,\cdot]$ is an Hilbert space. Note that we do not need to add a term $\RE \int_{\O_\e} \phi \overline{\varphi}$ in the inner product because we have a zero condition boundary for $\psi_2$ and a zero average condition for $\psi_1$, thus the Poincar\'e inequality yields that $[ \cdot, \cdot ]$ is an inner product equivalent as the usual one on $H$. We denote by $\langle k(x)\phi, \cdot \rangle$ the linear form  defined for every $\varphi$ in $H$ by
\begin{eqnarray}
-\langle k(x)\phi, \varphi \rangle= \RE \int_{\O_\e}\left[ (\eta-1)\frac{E}{V_0}+(1-|V_0|^2) \right]\phi \overline{\varphi} \nonumber \\ - \RE\int_{\O_\e} 2\RE(\overline{V_0}\phi) V_0 \overline{\varphi} -\RE \int_{\p \O_\e} \frac{\overline{V_0}\p_\nu V_0}{|V_0|^2} \phi \overline{\varphi}.
\end{eqnarray}
We also denote by $\langle s, \cdot \rangle$ the linear form defined by on $H$ by
\begin{equation}\nonumber
\langle s, \varphi \rangle=-\RE \int_{\O_\e} (R.H.S)iV_0 \overline{\varphi} \ \ \ \forall \varphi \in H.
\end{equation}
We can see that the variational formulation of the problem leads to
\begin{equation}
[\phi,\varphi ]+ \langle k(x)\phi, \varphi \rangle =\langle s ,\varphi \rangle \ \ \ \forall \varphi \in H. 
\end{equation}
We can then use Riesz's representation theorem to rewrite the problem \eqref{4.15}--\eqref{4.17} in the following operational form:
\begin{equation}\nonumber
\phi+K(\phi)=S
\end{equation}
for some $S$ in $H$ which depends linearly in $s$ and some operator $K$ defined on $H$. Furthermore we can check that $K$ is compact. Fredholm alternative then yields the existence assertion, provided that the homogeneous equation only has the trivial equation. But this is a direct consequence of Lemma \ref{lemma4.1} if we establish the a priori estimate \eqref{4.19}
%
For $1\leq j\leq k$, in the region $\{|y-\xi_j'|\leq \frac{\delta}{\e}\}$, the equation \eqref{4.15} is equivalent  to
\begin{equation}\label{4.20}
L_j^\e(\phi_j)=\tilde{h}_j+\sum_{j,l} c_{jl}\frac{1}{|\a_j|^2}w_{x_l}\chi_{\{|z|<1\}}.
\end{equation}
Here we have denoted $\tilde{h}_j(z)=iw(z)h(\xi_j'+z)$. Multiplying Equation \eqref{4.20} by $\overline{w}_{x_m}(y-\xi_j')$, integrating all over $B(0,\frac{\delta}{\e})$ and taking real parts one gets,
\begin{equation}\label{4.21}
\RE \int_{B(0,\frac{\delta}{\e})} L_j^\e(\phi_j) \overline{w}_{x_m}=\RE \int_{B(0,\frac{\delta}{\e})}\tilde{h}_j\overline{w}_{x_m}+c_{jm}c^* + O\left(\sum_{j,l}|c_{jl}|\e^2\right)
\end{equation}
where we have set
\begin{equation}\label{4.22}
c_*=\int_{B(0,1)}|w_{x_m}|^2 (\text{ this quantity is the same for } m=1 \text{ and } m=2),
\end{equation}
where we have used that $|\a_j|^2=1+O(\e^2)$ and that the elements $w_{x_m}$ are orthogonal to each other (cf.\ formula (1.17) in \cite{delPinoFelmerKowalczyk}).
The desired result will follow from estimating the left-hand side of equality \eqref{4.21}. Integrating by parts we obtain
\begin{eqnarray}\label{4.23}
\RE\int_{B(0,\frac{\delta}{\e})} L_j^\e(\phi_j)\overline{w}_{x_m}&= &\RE \left\{ \int_{\p B(0,\frac{\delta}{\e})} \p_\nu \phi_j \overline{w}_{x_m}-\int_{\p B(0,\frac{\delta}{\e})} \phi_j \frac{\p}{\nu}w_{x_m}\right\} \nonumber \\
& +& \RE \int_{B(0,\frac{\delta}{\e})}\overline{\phi}_j(L_j^\e-L^0)w_{x_m}.
\end{eqnarray}
Using that $\phi_j=iw(z)\psi(\xi_j'+z)$, that $w(z)=U(r)e^{i\theta}$ with $r=|z|$ and $\theta$ the angle around $\xi_j'$, the decay at infinity of $U(r), U'(r), U''(r)$ (cf.\ estimates \eqref{estimeessurlesderivees}, \eqref{estimeesurladeriveeseconde}) and the definition of the norm $\|\cdot\|_*$ we get that
\begin{equation}\nonumber
 \left|\RE \left\{ \int_{\p B(0,\frac{\delta}{\e})} \p_\nu \phi_j \overline{w}_{x_m}-\int_{\p B(0,\frac{\delta}{\e})} \phi_j \p_{\nu}w_{x_m}\right\}\right| \leq C\e \| \psi\|_*.
\end{equation}

Using the definition of the operator $L^\e_j$ \eqref{defL_j} and the estimates on $\alpha_j$ \eqref{estimatesonalpha_j}, the remaining term in \eqref{4.23} can be estimated in the following way:
\begin{eqnarray}
\RE \int_{B(0,\frac{\delta}{\e})} (L_j^\e-L^0)w_{x_m}\overline{\phi}_j=\RE \int_{B(0,\frac{\delta}{\e})} \left( \nabla \alpha_j \nabla w_{x_m}+  \Delta \alpha_j w_{x_m}+O(\e^2)w_{x_m}\right)\overline{\phi}_j.
\end{eqnarray}
Thus we get
\begin{eqnarray}
\left|\RE \int_{B(0,\frac{\delta}{\e})} (L^\e_j-L^0)w_{x_m}\overline{\phi}_j \right| &\leq& C  \int_1^{\delta/\e}\left( \frac{\e}{r^2}+\frac{\e^2}{r}rdr\right)\|\phi_j\|_\infty+\|\psi\|_*+O(\e^2) \nonumber \\
& \leq & C |\log \e| \| \psi\|_*.\nonumber
\end{eqnarray}
Combining the above estimates we obtain: 
\begin{eqnarray}
|c_{jm}| \leq \frac{1}{c*} \int_{B(0,\frac{\d}{\e})} |\tilde{h}_j| |\overline{w}_{x_m}| +O(\e \log \e) \|\psi\|_* +O\left( \sum_{jl} |c_{jl}| \e^2\right) \nonumber \\
\sum_{jl} |c_{jl}|  \leq \left[ \frac{1}{c*} \sum_{jl} \int_{B(0,\frac{\d}{\e})} |\tilde{h}_j| |\overline{w}_{x_m}| +O(\e\log \e) \|\psi\|_* \right] (1+O(\e^2)).
\end{eqnarray}
If $\|h\|_{**} <+\infty$ we can check that $ \int_{B(0,\frac{\d}{\e})} |\tilde{h}_j| |\overline{w}_{x_m}| \leq C\|h\|_{**}$ and hence we obtain
\begin{equation}
c_{jm} =-\frac{1}{c^*} \RE \int_{B(0,\frac{\d}{\e})} \tilde{h}_j\overline{w}_{x_m} +O(\e \log \e)+ O(\e^2) \|h\|_{**}.
\end{equation}
In particular it follows that
\begin{equation}\label{4.24}
|c_{jl}| \leq C \left[ \|h\|_{**}+\e |\log \e| \| \psi\|_*\right].
\end{equation}
Now we can apply Lemma \ref{lemma4.1} to get
\begin{equation}\label{4.25}
\| \psi \|_* \leq C\left[ |\log \e| \|h\|_{**}+|\log \e| |c_{jl}| +\|g\|_{***} \right].
\end{equation}
Estimate \eqref{4.18} then follows combining \eqref{4.24} and \eqref{4.25}.
\end{proof}
\section{The projected nonlinear problem}

Our goal is to solve problem \eqref{fullproblem} for a suitable $\psi$. We first consider its projected version, for $\xi$ in $\O_\delta^k$,
\begin{eqnarray}\label{5.1}
\mathcal{L}^\e(\psi)=N(\psi)+R+\sum_{j,l} c_{jl}\frac{1}{ iV_0\overline{\a_j}} w_{x_l}(x-\xi_j')\chi_{\{r_j<1\}} \nonumber \\
\phantom{aaaaaaaaaaaaaaaaaaaaaaaaaaaaaaaaaaaa} +c_0\e^2\chi_{\O_\e \setminus \cup_{j=1}^k B(\xi_j',\frac{\delta}{\e})} \ \text{ in } \O_\e,
\end{eqnarray}
\begin{equation}\label{5.2}
\p_\nu \psi_1=0, \ \ \ \ \ \ \ \psi_2 =-\log |V_0| \ \text{ on } \p \O_\e,
\end{equation}
\begin{eqnarray}\label{5.3}
\int_{\O_\e \setminus \cup_{j=1}^k B(\xi_j',\frac{\delta}{\e})} \psi_1=0,  \ \ \ \ \RE \int_{|y|<1} \overline{\phi}_jw_{x_l}=0, \ \forall \ j,l, \nonumber \\
\phi_j(z)=iw(z)\psi(\xi_j'+z).
\end{eqnarray}
\begin{proposition}\label{Prop5.1}
There is a constant $C>0$ depending only on $\delta$ and $\O$ such that for all points $\xi \in \O_\delta^k$ and $\e$ small, problem \eqref{5.1}--\eqref{5.3} possesses a unique solution with
\begin{equation}\nonumber
\|\psi\|_* \leq C \e^{1-\sigma}.
\end{equation}
Moreover, we have that 
\begin{equation}\label{formulac0}
c_0\e^2\int_{\O_\e \setminus \bigcup_{j=1}^k B(\xi_j',\frac{\d}{\e})}1 = \sum_{j,l} c_{jl}\IM \int_{\{r_j<1\}} \overline{\phi_j}w_{x_l}.
\end{equation}
\end{proposition}
\begin{proof}
The boundary condition for $\psi$ are $\p_\nu \psi_1=0$ and $\psi_2=-\log |V_0|$. We can see, with the help of Lemma \ref{Lemma2.1},  that
\begin{equation}\nonumber
\|\log |V_0|\|_\infty=O(\e^2), \ \ \ \ \|\nabla \log |V_0|\|_\infty=O(\e^3) \ \ \ \text{ and } [\nabla \log |V_0|]_{\gamma,\p \O_\e} =O(\e^4) 
\end{equation}
on $\p \O_\e$ (see Lemma \ref{Lemma2.1}). Thus we have $\| \log |V_0| \|_{***}=O(\e^{1-\sigma})$. As for the error $R=R_1+iR_2$, Lemma \ref{Lemma2.1} yields
\begin{equation}\nonumber
R_1=O\left(\e\sum_{j=1}^k \frac{1}{r_j^3}\right), \ \ \ \ R_2=O\left(\e \sum_{j=1}^k \frac{1}{r_j}\right)
\end{equation}
if $r_j>1$ for all $1\leq j \leq k$. Calling $\tilde{R}_j$ the error in $\phi_j$-coordinates (see \eqref{defRj}) we also find
\begin{equation}\nonumber
\| \tilde{R}_j\|_{L^p(|z|<3)}=O(\e),
\end{equation}
and then we conclude
\begin{equation}\nonumber
\|R\|_{**} \leq C\e^{1-\sigma}.
\end{equation}
Here and in what follows $C$ denotes a generic constant independent of $\e$. We make the following claim: if $\|\psi\|_* \leq C\e^{1-\sigma}$ then $\|N(\psi)\|_{**}\leq C \e^{2-2\sigma}$. In fact, for $r_j>2$ for all $j$, $N(\psi)$ reduces to
\begin{equation}\nonumber
N(\psi)_1=-2\nabla \psi_1 \nabla \psi_2, \ \ \ N(\psi)_2=|\nabla \psi_1|^2-|\nabla \psi_2|^2+|V_0|^2(e^{-2\psi_2}+1-2\psi_2)
\end{equation}
(see \eqref{pblineariseloindesvortex}). The definitions of the $*$-norm yields that in this region
\begin{equation}\nonumber
|N(\psi)_1|\leq C\e^{2-2\sigma}\frac{1}{r_j^{2+\sigma}}, \ \ \ |N(\psi)_2|\leq C \e^{2-2\sigma}\frac{1}{r_j^2}.
\end{equation}
On the other hand, calling $\tilde{N}_j(\phi_j)$ the operator in the $\phi_j$-variable, as defined in \eqref{defNj} we see that
\begin{eqnarray}\nonumber
\tilde{N}_j(\phi_j) & =& A_1(z,\phi_j,\nabla \phi_j)+(1-\tilde{\eta}) w \Delta \left[ e^{\frac{\phi_j}{w}}-1-\frac{\phi_j}{w}\right ] \nonumber \\
&=& A_1(z,\phi_j,\nabla \phi_j) + A_2(z,\phi_j,\nabla \phi_j) +(1-\tilde{\eta}) \Delta \phi_j (e^{\frac{\phi_j}{w}}-1)
\end{eqnarray}
where $A_i$ are smooth functions of their arguments, with $A_2$ supported only for $|z|>1$, and with
\begin{equation}
|A_1(z,p,q)|\leq C\left[ |p|^2+|q|^2\right], \ \ \ |A_2(z,p,q)| \leq C\left[ |p|^2+|q|^2\right]
\end{equation}
near $(p,q)=(0,0)$. By assumption we have
\begin{equation}\nonumber
\| \phi_j\|_{W^{2,p}(|z|<2)} \leq C \e^{1-\sigma},
\end{equation}
we then use the Sobolev injection $W^{2,p} \hookrightarrow C^1$ for $p>2$ to deduce that $\|\phi_j\|_{L^\infty(|z|<3)} \leq C\e^{1-\s}$ and $\|\nabla \phi_j\|_{L^\infty(|z|<3)} \leq C\e^{1-\s}$. We then have 
\[ \|A_1(z,\phi_j,\nabla \phi_j) +A_2(z,\phi_j,\nabla \phi_j) \|_{L^p(|z|<3)} \leq C\e^{2-2\s}. \]
But since $|e^{x}-1|\leq x$ for $x$ near zero, we also have 
\begin{eqnarray}
\| (1-\tilde{\eta}) \Delta \phi_j (e^{\frac{\phi_j}{w}}-1) \|_{L^p(|z|<3)} & \leq & \|\frac{\phi_j}{w} \|_{L^\infty (1<|z|<3)} \| \Delta \phi_j \|_{L^p(|z|<3)} \nonumber \\
 & \leq & C \e^{1-\s}\times\e^{1-\s} \\
 & \leq & C \e^{2-2\s}.
\end{eqnarray}

On the other hand, it is also true that if $\|\psi^l\|_* \leq C \e^{1-\sigma}$ for $l=1,2$ then
\begin{equation}\nonumber
\| N(\psi^1)-N(\psi^2)\|_{**} \leq C\e^{\frac{1-\sigma}{2}}\| \psi^1-\psi^2\|_*.
\end{equation}
Problem \eqref{5.1}--\eqref{5.3} is equivalent to the fixed point problem
\begin{equation}\nonumber
\psi=T_\e(N(\psi)+R, \log |V_0|),
\end{equation}
where $T_\e$ is the linear operator introduced in Proposition \ref{Prop4.1}. Since $\|T_\e\|=O(\log \e)$, the above estimates yield a unique solution with size $\|\psi\|_* \leq C\e^{1-\sigma}$. Hence we have proven the existence of a unique solution in this range for problem \eqref{5.1}--\eqref{5.3}.

We now prove the formula \eqref{formulac0} for $c_0$. If $\psi$ satisfies \eqref{5.1}--\eqref{5.3} then the ansatz $v$ given by \eqref{ansatz} satisfies
\begin{equation}\nonumber
\Delta v+\left( 1-|v|^2\right)v=c_0\e^2iv\chi_{\text{out}}+ \sum_{j,l} c_{jl}\frac{1}{\overline{\a_j}}w_{x_l}(y-\xi_j')\chi_{\{r_j<1\}},
\end{equation}
where $\chi_{out}=\chi_{\O_\e \setminus \cup_{j=1}^k B(\xi_j',\frac{\delta}{\e})}$. In $\O_\e \setminus \cup_{j=1}^k B(\xi_j',\frac{\delta}{\e})$. Indeed we used that the ansatz take the form $v=V_0e^{i\psi}$ in the region $\O_\e \setminus \cup_{j=1}^k B(\xi_j',\frac{\delta}{\e})$ and that 
\begin{eqnarray}
-\Delta v +(1-|v|^2)v =  iV_0\left [ \mathcal{L}^\e(\psi) -R +N(\psi) \right]e^{i\psi}
\end{eqnarray} in that region. We also used that 
\begin{eqnarray}
-\Delta v +(1-|v|^2)v =  iV_0\left [ \mathcal{L}^\e(\psi) -R +N(\psi) \right]
\end{eqnarray} in the complement of that region. 
%

Multiplying the above relation by $\overline{v}$, using that for $r_j<1$ we have that $v(\xi_j'+z)=\alpha_j(z)\left[ w(z)+\phi_j(z)\right]$ and integrating we get that
\begin{equation}\label{equality0}
-\int_{\O_\e}|\nabla v|^2+\int_{\O_\e}(1-|v|^2)|v|^2+\int_{\p \O_\e} \overline{v} \p_\nu v= \mathcal{R}
\end{equation}
with
\begin{eqnarray}
\mathcal{R}:= ic_0\e^2\int_{\O_\e} |v|^2\chi_{\text{out}}
+\sum_{j,l}c_{jl}\int_{\O_\e}\left( \overline{w}w_{x_l}+\overline{\phi_j}w_{x_l}  \right)\chi_{\{r_j<1\}}.\nonumber
\end{eqnarray}
We observe that $\int_{\{r_j<1\}} \overline{w} w_{x_l}=0$ due to the form of $w$ and $w_{x_l}$. Now using the boundary condition $v\wedge \p_\nu v=\IM (\overline{v}\p_\nu v)=0$ on $\p \O_\e$ we see that the left hand side of \eqref{equality0} is real-valued. Thus we must have 
\begin{equation}
\e^2\int_{\O_\e \setminus \bigcup_{j=1}^k B(\xi_j',\frac{\d}{\e})} c_0= \sum_{j,l} c_{jl}\IM \int_{\{r_j<1\}} \overline{\phi_j}w_{x_l}.
\end{equation}
\end{proof}
%

As explained in \cite{delPinoKowalczykMusso2006}, the function $\psi(\xi)$ turns out to be continuously differentiable.
We have that a solution $v(\xi)$ given by Proposition \ref{Prop5.1} is a solution of our problem if and only if the constants $c_{jl}$ are equals to zero. We thus need to adjust $\xi$ in $\mathcal{D}$ in such a way that $c_{jl}=0$ for all $j,l$ in \eqref{5.1}--\eqref{5.3}. We will see that this problem is equivalent to a variational problem which is very close to the one of finding critical points of the renormalized energy. In the conclusion we give the expression of the renormalized energy as computed in \cite{delPinoKowalczykMusso2006}, we formulate the variational reduction and we indicate that the rest of the proof follows exactly the same line as in \cite{delPinoKowalczykMusso2006}.

\section{Conclusion}
Let $\Gamma_0$ be the outer component of $\p \O$, and let us denote by $\Gamma_l, \ l=1,...,n$, its inner components, if any. Let us call $\phi_l(x)$ the solution of the following problem
\begin{equation}\nonumber
\left\{
\begin{array}{rclll}
\Delta \phi_l&=&0 & \text{ in } \O, \\
\phi_l &=&\delta_{lj} & \text{ on } \Gamma_j, \ \forall j=1,...k.
\end{array}
\right.
\end{equation}

Let $G_0(x,\xi)$ denote the Green's function for the problem
\begin{equation}\nonumber
\left\{
\begin{array}{rclll}
-\Delta G_0&=&2\pi \delta_{\xi_j} & \text{ in } \O, \\
G_0(x,\xi)&=&0 & \text{ on }  \p \O,
\end{array}
\right.
\end{equation}
and $H_0(x,\xi)$ its regular part,
\begin{equation}\nonumber
H_0(x,\xi) =\log \frac{1}{|x-\xi|}-G_0(x,\xi).
\end{equation}
We set $\gamma_l:= 2\pi \left( \int_{\Gamma_l} \p_\nu \phi_l\right)^{-1}$ and we let
\begin{equation}
G(x,\xi)=\sum_{l=1}^n \gamma_l \phi_l(\xi)\phi_l(x)+G_0(x,\xi),
\end{equation}
where the sum is understood to be zero if the domain is simply connected. Consistently we let
\begin{equation}\nonumber
H(x,\xi)=-\sum_{l=1}^n \gamma_l \phi_l(\xi)\phi_l(x)+H_0(x,\xi).
\end{equation}
Then we have (cf.\ \cite{delPinoKowalczykMusso2006})
\begin{equation}\nonumber
W_N(\xi,d)=\pi \sum_{i\neq j} d_i d_jG(\xi_i, \xi_j)-\sum_{j=1}^k\pi H(\xi_j,\xi_j).
\end{equation}
Calling $w_{N\e}$ the same function as in \eqref{approximation}  we have
\begin{eqnarray}\label{renormalizede}
E_\e(w_{N_\e}(\cdot,\xi,\mathtt{d}))&=&k\pi \log \frac{1}{\e}+W_N(\xi,\mathtt{d})+c+O(\e) \nonumber \\
\nabla_\xi E_\e(w_{N_\e}(\cdot,\xi,\mathtt{d}))&=&\nabla_\xi W_N(\xi,\mathtt{d})+O(\e) \nonumber
\end{eqnarray}
with $c$ a constant which depends on the number $k$ of points. \\

Now we consider the equations $c_{jl}(\xi)=0$ in \eqref{5.1}--\eqref{5.3} for the solution $\psi=\psi(\xi)$ predicted by Proposition . We denote by $v(\xi)$ the ansatz for this $\psi$ and consider the functional
\begin{equation}
P_\e(\xi)=E_\e(v(\xi)).
\end{equation}

As in \cite{delPinoKowalczykMusso2006} we can see that solving $c_{jl}(\xi)=0$ for all $j,l$ is equivalent to finding critical points of $P_\e$. Besides $P_\e$ is close to the renormalized energy in a $C^1$-sense. 
\begin{proposition}
\begin{itemize}
\item[a)] If $\nabla_\xi P_\e(\xi)=0$ then $c_{jl}(\xi)=0$ for all $j,l$, and hence $c_0=0$.
\item[b)] We have the validity of the expansion
\begin{equation}
\nabla_\xi P_\e(\xi)=\nabla_\xi W_N(\xi,d)+O(\e^{1-\sigma}\log \e),
\end{equation}
uniformly on $\xi$ in $\O_\delta^k$.
\end{itemize}
\end{proposition}

\begin{proof}
We let $\xi=(\xi_1,...,\xi_k)$, $\xi_j=(\xi_{j1},\xi_{j2})$ and $\xi_j'=\frac{\xi}{\e}=(\xi_{j1}',\xi_{j2}')$. We then have
\begin{eqnarray}
-\p_{\xi'_{j_0i_0}} P_\e(\xi)&=&-J_\e'(v(\xi)). \p_{\xi'_{j_0i_0}}v \nonumber \\
 &=& \RE \int_{\O_\e} -\nabla v \nabla \overline{v}_{\xi'_{j_0i_0}}+(1-|v|^2)v\overline{v}_{\xi'_{j_0i_0}} \nonumber \\
 &=& \RE \int_{\O_\e} \left( \Delta v+(1-|v|^2)v \right)\overline{v}_{\xi'_{j_0i_0}}- \RE \int_{\p\O_\e} \p_\nu v \overline{v}_{\xi'_{j_0i_0}}. \nonumber 
\end{eqnarray}
Now since $|v(\xi)|=1$ on $\p \O_\e$ for all $\xi$ in $\O_\d$ , we have $\RE (v\p_{\xi'_{j_0i_0}} 
\overline{v})=0$ on $\p \O_\e$, geometrically $v$ and $\p_{\xi'_{j_0i_0}} v$ are orthogonal on $\p \O_
\e$. But by using the boundary condition $v \wedge \p_\nu v=0$ on $\p \O_\e$ ($v$ and $\p_\nu v$ are 
parallel on $\p \O_\e$) we find $\RE ( \p_\nu v \overline{v}_{\xi_{j_0i_0}})=0$ on $\p \O_\e$. Thus 
\begin{eqnarray}
-\p_{\xi'_{j_0i_0}} P_\e(\xi)&=& \RE \int_{\O_\e} \left( \Delta v+(1-|v|^2)v \right)\overline{v}_{\xi'_{j_0i_0}} \nonumber \\
&=& \RE  \left( ic_0 \e^2 \int_{\O_\e \setminus \bigcup_{j=1}^k B(\xi_j',\frac{\d}{\e})} v \overline{v}_{\xi'_{j_0l_0}} \right)+ \sum_{l,j} c_{jl} \RE \int_{\{|z|<1 \}} \frac{1}{\overline{\a_j}} w_{x_l}(z) \overline{v}_{\xi'_{j_0i_0}}(\xi_j'+z). \nonumber
\end{eqnarray}

Now near $\xi_j'$ we have
\begin{eqnarray}
\p_{\xi'_{j_0i_0}} v(y)&=& \p_{\xi'_{j_0i_0}} \left[ \alpha_j(y-\xi_j',\xi) \left( w(y-\xi_j')+\phi_j(y-\xi_j',\xi) \right) \right ] \nonumber \\
&=& \left( \p_{\xi'_{j_0i_0}} \alpha_j\right)(w+\phi)+\alpha_j (\p_{\xi'_{j_0i_0}}\phi_j) -\d_{jj_0} \p_{z_{i_0}}(\alpha_j w+\alpha_j \phi_j).
\end{eqnarray}
We observe that $\p_z \alpha_j$ and $\p_{\xi'} \alpha_j$ are of order $O(\e)$ in $\{|z|<1\}$ (cf.\ \eqref{estimatesonalpha_j}). Besides $\phi_j$ and $\p_z \phi_j$ are of order $O(\e^{1-\s})$ in $L^\infty$ norm. To see that we use the Sobolev injection $W^{2,p} \hookrightarrow C^1$ and $\|\phi_j\|_{W^{2,p}(|z|<3)}\leq C\e^{1-\s}$. We also know that \[\RE \int_{B(0,1)}\phi_j(z,\xi) \overline{w}_{x_l}(z) dz=0  \text{ for all }\xi \text{ in } \O_\d,\] thus \[ \RE \int_{B(0,1)}\p_{\xi'_{j_0i_0}}\phi_j(z,\xi) \overline{w}_{x_l}(z) dz=0.\] Besides 
\[ \RE \int_{B(0,1)} \overline{w}_{z_{i_0}} w_{x_l}=c^*\d_{i_0l} \text{ with } c^*=\int_{B(0,1)} |w_{z_1}|^2dz. \]

We thus find
\[-\p_{\xi'_{j_0i_0}}P_\e(\xi)=\RE \left( ic_0 \e^2 \int_{\O_\e \setminus \bigcup_{j=1}^k B(\xi_j',\frac{\d}{\e})} v \overline{v}_{\xi'_{j_0l_0}} \right)+\sum_{l,j} c_{jl} \left(c^* \d_{jj_0}\d_{i_0 l} +O(\e^{1-\s}) \right).\]
Now we use that $v=V_0e^{i\psi}$ in $\O_\e \setminus \bigcup_{j=1}^k B(\xi'_j,\frac{\d}{\e})$ and we get that
\begin{eqnarray}
\IM  \int_{\O_\e \setminus \bigcup_{j=1}^k B(\xi_j',\frac{\d}{\e})} v \overline{v}_{\xi'_{j_0i_0}} = \IM \int_{\O_\e \setminus \bigcup_{j=1}^k B(\xi_j',\frac{\d}{\e})}  V_0 \p_{\xi'_{j_0i_0}} \overline{V_0}  -i |V_0|^2 \p_{\xi'_{j_0i_0}} \overline{\psi} . \nonumber
\end{eqnarray}
We first observe that (see formula \eqref{gradientV0}) 
\begin{eqnarray}
\RE  \int_{\O_\e \setminus \bigcup_{j=1}^k B(\xi_j',\frac{\d}{\e})} V_0 \p_{\xi'_{j_0i_0}} \overline{V_0} =O(\e)
\end{eqnarray}
We also have  
\begin{eqnarray}
\RE  \int_{\O_\e \setminus \bigcup_{j=1}^k B(\xi_j',\frac{\d}{\e})} |V_0|^2 \p_{\xi'_{j_0i_0}}\overline{ \psi}  &=& \RE  \int_{\O_\e \setminus \bigcup_{j=1}^k B(\xi_j',\frac{\d}{\e})} |V_0|^2 \p_{\xi'_{j_0i_0}}\psi_1 \nonumber \\
&=&  \int_{\O_\e \setminus \bigcup_{j=1}^k B(\xi_j',\frac{\d}{\e})} \left( 1+O(\e^2)\right) \p_{\xi'_{j_0i_0}} \psi_1. \nonumber
\end{eqnarray}
Now we use that since 
\begin{equation}
\int_{\O_\e \setminus \bigcup_{j=1}^k B(\xi_j',\frac{\d}{\e})} \psi_1 =0 \ \ \ \ \forall \xi \text{ in } \O_\delta
\end{equation}
we have that 
\begin{equation}
\int_{\O_\e \setminus \bigcup_{j=1}^k B(\xi_j',\frac{\d}{\e})} \p_{\xi'_{j_0i_0}} \psi_1- \int_{B(\xi'_{j_0},\frac{\d}{\e})}\p_{z_{i_0}}\psi_1=0.
\end{equation}
By using that $\|\psi\|_* \leq C\e^{1-\s}$ we obtain
\begin{equation}
\e^2 \int_{\O_\e \setminus \bigcup_{j=1}^k B(\xi_j',\frac{\d}{\e})} |V_0|^2 \p_{\xi'_{j_0i_0}}\psi_1=O(\e^{2-\sigma}).
\end{equation}
We can thus say that
\begin{equation}
-\p_{\xi'_{j_0i_0}}P_\e(\xi)= c_0O(\e^{1-\s})+\sum_{l,j} c_{jl} \left(c^* \d_{jj_0}\d_{i_0 l} +O(\e^{1-\s}) \right).
\end{equation}

Now we recall that 
\begin{equation}
\e^2\int_{\O_\e \setminus \bigcup_{j=1}^k B(\xi_j',\frac{\d}{\e})} |v|^2 c_0= \sum_{j,l} c_{jl}\IM \int_{\{r_j<1\}} \overline{\phi_j}w_{x_l}
\end{equation}
ans since we can prove that $|v|=1+O(\e^2)$ we see that we have $\e^2\int_{\O_\e \setminus \bigcup_{j=1}^k B(\xi_j',\frac{\d}{\e})} |v|^2 $ is of order 1. Hence By using again that $\|\psi\|_* \leq C\e^{1-\s}$ we arrive at
\begin{eqnarray}
\p_{\xi'_{j_0i_0}} P_\e(\xi)&=&-\sum_{l,j} c_{jl} \left(c^* \d_{j j_0}\d_{l i_0} +O(\e^{1-\s}) \right)  \nonumber \\
&=& -c_{j_0i_0}c^*- \sum_{j,l} c_{j,l} O(\e^{1-\s}).
\end{eqnarray}
From this last equality we can deduce that if $\nabla _\xi P_\e(\xi)=0$ then $c_{jl}=0$ for all $j=1,...k$ $l=1,2$. Indeed by contradiction if we assume that $\nabla _\xi P_\e(v(\xi))=0$ and that there exists $c_{j_0l_0} \neq 0$ then we find that $|c_{j_0l_0}|=O(\sum_{j,l} |c_{jl}| \e^2)$ and by adding all the non-zero terms we arrive at $\sum_{jl}|c_{jl}|(1+O(\e^2))=0$ which is a contradiction.
 
To prove point b) we remark that from Proposition \ref{Prop4.1} 
\begin{equation}
c^*c_{jl}=-\RE \int_{\{ |z|<\frac{\d}{\e} \}} iw (N(\psi)+R) \overline{w}_{x_l} +O(\e^{2-\s}\log \e).
\end{equation}
But 
\[\int_{\{ |z|<\frac{\d}{\e} \}} iw N(\psi) \overline{w}_{x_l} = \int_{\{ |z|<2 \}}i N(\psi) \overline{w}_{x_l} +\int_{\{ 2<|z|< \frac{\d}{\e} \}}i N(\psi) \overline{w}_{x_l}\]
Hence by using that $\|N(\psi)\|_{**} \leq C\e^{2-2\s}$ the definition of the norm-** and the fact that $|w_{x_l}|\leq C\frac{1}{|z|}$ for $|z|>2$ we obtain: 
\[ \RE \int_{\{ |z|<\frac{\d}{\e} \}} iw N(\psi)\overline{w}_{x_l}  =O(\e^{2-\s}) .\]
Now since $w=\frac{V_0}{\a_j}$ we have that 
\begin{equation}
\RE \int_{\{|z|<\frac{\d}{\e}\}} iw R \overline{w}_{x_l} =\RE \int_{\{|z|<\frac{\d}{\e}\}} \frac{E^j}{\alpha_j}\overline{w}_{x_l}
\end{equation}
with $E^j= \Delta V_0^j+(1-|V_0^j|^2)V_0^j$ and $V_0^j=V_0(\xi_j'+z)$.  We also have that 
\begin{equation}
\p_{x_l} w= \frac{\p_{x_l} V_0^j}{\a_j}-\frac{V_0^j \p_{x_l} \a_j}{\a_j^2}.
\end{equation}
Now we have that $\p_{x_l}V_0^j= \p_{\xi'_{jl}} V_0$ and 
\begin{equation}\nonumber
\RE \int_{\{|z|<\frac{\d}{\e}\}} iw R \overline{w}_{x_l}=\RE \int_{\{|z|<\frac{\d}{\e}\}}\left( \Delta V_0+(1-|V_0|^2)V_0 \right)\p_{\xi'_{jl}}V_0 (1+O(\e^2)) +O(\e^{2-2\s}).
\end{equation}
To obtain the last equality we used that $|\a_j|^2=1+O(\e^2)$ along with the form and the estimates on $\alpha$ (cf.\ \eqref{defalpha3}, \eqref{estimateonA}, \eqref{estimateonB}) and $E=iV_0[R_1+iR_2]$ with estimates on $R_1$ and $R_2$. Now estimates on $E$ and on the gradient of $V_0$ give that 
\begin{eqnarray}
\RE \int_{\{|z|<\frac{\d}{\e}\}} \left( \Delta V_0+(1-|V_0|^2)V_0 \right)\p_{\xi'_{jl}}V_0 &=&\RE \int_{\O_\e}  \left( \Delta V_0+(1-|V_0|^2)V_0 \right)\p_{\xi'_{jl}}V_0 +O(\e^{2-\s}) \nonumber \\
&=&-\p_{\xi'_{jl}} E_\e(V_0)+O(\e^{2-\s}).
\end{eqnarray}
We thus obtain that 
\[c_{jl}c^*= \p_{\xi'_{jl}} E_\e(V_0)+O(\e^{2-\s}\log \e).\]
This proves that 
\[\p_{\xi'_{jl}} P_\e(\xi)=\p_{x'_{jl}} E_\e(V_0)+O(\e^{2-\s}\log \e) \]
but according to \eqref{renormalizede} 
\[ \p_{\xi'_{jl}} E_\e(V_0) =\e \p_{\xi_{jl}} W_N(\xi,d) +O(\e^{2-\s}) \]
hence
\begin{equation}
\nabla_\xi P_\e(\xi)=\p_{\xi_{jl}} W_N(\xi,d)+O(\e^{1-\s}\log \e).
\end{equation}
\end{proof}
Once we have this Proposition the proof of Theorem \ref{theorem1} and Theorem \ref{theorem2} follows exactly as in \cite{delPinoKowalczykMusso2006} and we refer the reader to this article for a detailed proof.\\

\textbf{Acknowledgments:} I would like to thank Juan D\'avila for his help and for many useful conversations about this work. I am also grateful to Monica Musso for valuable comments. Many thanks also to Xavier Lamy for interesting conversations about this work. This work has been supported by the Millennium Nucleus Center for Analysis of PDE NC130017 of the Chilean Ministry of Economy. 

\section*{Appendix}
We recall the following useful elliptic estimates (see \cite{BBH0}) 
\begin{proposition}\label{EE1}
Let $u$  be a solution of 
\[ -\Delta u= f \ \text{ in } \O. \]
Then 
\begin{equation}
|\nabla u(x)|^2 \leq C\left( \|f\|_{L^\infty(\O)} \|u\|_{L^\infty(\O)}+\frac{1}{\text{dist}^2(x,\p \O)\|u\|_{L^\infty(\O)}^2} \right) \ \text{ for all } x \text{ in } \O
\end{equation}
where $C$ is some constant depending only on $N$ with $\O \subset \R^N$.
\end{proposition}

\begin{proposition}\label{EEE2}
Let $u$  be a solution of 
\begin{equation}
\left\{
\begin{array}{rcll}
-\Delta u &=& f \ \text{ in } \O \\
u&=&0 \ \text{ on } \p\O.
\end{array}
\right.
\end{equation} 
Then 
\begin{equation}
||\nabla u(x)\|_{L^\infty(\O)}^2 \leq C \|f\|_{L^\infty(\O)} \|u\|_{L^\infty(\O)}
\end{equation}
where $C$ is some constant depending only on $\O$.
\end{proposition} 

\begin{proposition}\label{EE3}
Let $u$  be a solution of 
\begin{equation}
\left\{
\begin{array}{rcll}
-\Delta u &=& f \ \text{ in } \O \\
u&=&g \ \text{ on } \p\O,
\end{array}
\right.
\end{equation} 
with $g$ in $C^{1,\gamma}(\p \O)$. Then 
\begin{equation}
||\nabla u(x)\|_{L^\infty(\O)}^2 \leq C \left( \|f\|_{L^\infty(\O)} \|u\|_{L^\infty(\O)} +\|g\|^2_{C^{1,\gamma}(\p \O)} \right).
\end{equation}
where $C$ is some constant depending only on $\O$.
\end{proposition}

\bibliographystyle{plain}
\bibliography{GLsemistiff}

\end{document}